\documentclass[12pt,a4paper,utf8x]{article}
\usepackage{color}
\usepackage{fancyhdr}
\usepackage{tikz}
\usepackage[vmargin=100pt, hmargin=60pt]{geometry}
\usetikzlibrary{decorations.pathmorphing}
\usepackage[T1]{fontenc}
\usepackage{ucs}
\usepackage[utf8x]{inputenc}
\usepackage{amsfonts}
\usepackage{listings}
\usepackage{amsmath}
\usepackage[mathscr]{euscript}
\usepackage{amssymb}
\usepackage{amsfonts}
\usepackage{dsfont}
\usepackage{amsthm}
\usepackage{fancybox}
\usepackage{enumitem}
\usepackage{dsfont}
\usepackage{algpseudocode}
\usepackage{algorithm}
\usepackage{graphicx}
\usepackage{caption}
\usepackage{subcaption}
\usepackage{float}
\usepackage{authblk}
\usepackage{relsize}
\usepackage{setspace}

\DeclareMathOperator*{\argsup}{argsup}

\everydisplay={\aftergroup\specindent}
\def\specindent{\global\hangindent=2em \global\hangafter=-1 \global\prevgraf=0 }

\makeatletter
\newcommand*{\inlineequation}[2][]{%
  \begingroup
    \refstepcounter{equation}%
    \ifx\\#1\\%
    \else
      \label{#1}%
    \fi
    \relpenalty=100 %
    \binoppenalty= 100 %
    \ensuremath{%
      #2%
    }%
    ~\@eqnnum
  \endgroup
}
\makeatother

\usepackage{hyperref}

\newtheorem{definition}{Definition}[section]
\newtheorem{theorem}{Theorem}[section]
\newtheorem{proposition}{Proposition}[section]
\newtheorem{lemme}{Lemma}[section]
\newtheorem{coro}{Corollary}[section]
\newtheorem{remark}{Remark}[section]

\date{}

\usepackage{todonotes}

\begin{document}

\title{Arbitrary order total variation Convergence of Markov semigroups using random grids}
\author{C. Rey\thanks{CMAP, Ecole Polytechnique. Email: \texttt{clement.rey@polytechnique.edu}.
\protect }}

\maketitle


\begin{abstract}
We provide a abstract framework to prove total variation convergence result with arbitrary rate for numerical scheme for SDE. In particular we show that under standard weak approximation properties of scheme such as Euler we can obtain total variation convergence with any desired rate by building a specific approximation for the SDE.
\tableofcontents{}
\end{abstract}


\section{The distance between Semigroups}
\label{Section:The distance between two Markov semigroups}
Throughout this section the following notations will prevail. We fix $T>0$ and $n \in \mathbb{N}^*$. For $l \in \mathbb{N}$ we will denote $\delta_{n}^{l}=T/n^{l}$ and for $\delta>0$ we consider the time grid $\pi^{\delta}:= \{k \delta, k \in \mathbb{N}\}$, with the convention $\pi^0=\mathbb{R}_+$. 

\subsection{Framework}
\paragraph{Notations.} 
For $d \in \mathbb{N}^{\ast}$, denote by
\begin{itemize}
\item $\mathcal{M}_b(\mathbb{R}^d)$, the set of measurable and bounded functions from $\mathbb{R}^d$ to $\mathbb{R}$.
\item $\mathcal{C}_b^q(\mathbb{R}^d) $, $q \in \mathbb{N} \cup \{+\infty\}$, the set of functions from $\mathbb{R}^d$ to $\mathbb{R}$ which admit derivatives up to order $q$ and such that all those derivatives are bounded.
\item $\mathcal{C}_1^q(\mathbb{R}^d) $, $q \in \mathbb{N} \cup \{+\infty\}$, the set of functions from $\mathbb{R}^d$ to $\mathbb{R}$ which admit derivatives up to order $q$ and such that all those derivatives are bounded in $\mbox{L}_1\left(\mathbb{R}^d\right)$.
\item $\mathcal{C}_c^q(\mathbb{R}^d) $, $q \in \mathbb{N} \cup \{+\infty\}$, the set of functions from $\mathbb{R}^d$ to $\mathbb{R}$ defined on compact support and which admit derivatives up to order $q$.
\end{itemize}

\paragraph{Discrete semigroups.}

We are going to introduce the definition of a discrete semigroup, as a sequence of functional operator $(Q_{t})_{t \in \pi^{\delta}}$, for $\delta>0$, built from a sequence of finite positive transition probability measures $(\mu^{\delta} _{t}(x,dy))_{t \in \pi^{\delta}}$
from $\mathbb{R}^{d}$ to $\mathbb{R}^{d}.$ This means that for each fixed $x$ and $t$, $\mu^{\delta} _{t}(x,dy)$
is a finite positive measure on $\mathbb{R}^{d}$ with the borelian $\sigma $ field and for
each bounded measurable function $f:\mathbb{R}^{d}\rightarrow \mathbb{R}$, the application 
\begin{align*}
x\mapsto \mu^{\delta} _{t}f(x):=\int_{\mathbb{R}^d} f(y)\mu^{\delta} _{t}(x,dy) 
\end{align*}
is measurable. We also denote 
\begin{align*}
\vert \mu^{\delta} _{t}\vert :=\sup_{x\in \mathbb{R}^{d}}\sup_{\Vert
f\Vert _{\infty }\leqslant 1} \big \vert \int_{\mathbb{R}^d} f(y)\mu^{\delta} _{t}(x,dy) \big\vert
,
\end{align*}
and, we assume that all the sequences of measures we consider in this paper satisfy:
\begin{align}
\label{eq:borne_sup_module_mesure_VT}
\underset{k \in \mathbb{N}^{\ast}}{\sup}\vert \mu^{\delta} _{t}\vert < \infty.
\end{align} 
Now we define the discrete semigroup associated to the sequence of measures $\mu^{\delta}$ to the time grid $\pi^{\delta}$.

\begin{definition}
\label{def:Semigroup_P}
Let $\delta >0$ and $(\mu^{\delta} _{t})_{t \in \pi^{\delta}}$ be a sequence of transition measures from $\mathbb{R}^{d}$ to $\mathbb{R}^{d}$ and satisfying (\ref{eq:borne_sup_module_mesure_VT}). The discrete semigroup $(Q^{\delta}_t)_{t \in \pi^{\delta}}$ associated to $\mu^{\delta}$ is defined by: For every $t \in \pi^{\delta}$ and $f \in \mathcal{M}_b(\mathbb{R}^d)$,
\begin{align*}
Q^{\delta}_{0}f(x)=f(x),\qquad Q^{\delta}_{t+\delta}f(x)=Q^{\delta}_{t}\mu^{\delta} _{t}f(x)=Q^{\delta}_{t} \int_{\mathbb{R}^d} f(y)\mu^{\delta} _{t}(x,dy). 
\end{align*}%

More generally, we define $(Q^{\delta}_{s,t})_{s,t \in \pi^{\delta} ;t \leqslant s}$ by: For every $t \in \pi^{\delta} $
\begin{align*}
Q^{\delta}_{t,t}f(x)=f(x),\qquad \forall  r \in \mathbb{N}, \; Q^{\delta}_{t,t+(r+1)\delta}f(x)=Q^{\delta}_{t,t+r\delta} \mu^{\delta} _{t+r\delta} f(x). 
\end{align*}%
\end{definition}

%
%

We notice that for $t,s,u \in \pi^{\delta}$, $t \leqslant s \leqslant u$, we have the semigroup property $Q_{t,u}f=Q_{t,s}Q_{s,u}f $. \\
Moreover, if for every $f \in \mathcal{C}^{\infty}_c\left(\mathbb{R}^d\right)$, $Pf=Qf$ then for every measurable and bounded function $f$ we have $Pf=Qf$.

\paragraph{Framework.}

%
The approach we propose consists in building an approximation for a family of semigroups $\left(P^{\delta}_{t,s})_{t ,s \in \pi^{\delta} ;t \leqslant s} \right)_{\delta >0}$ where we suppose that this family is independent to the time-grid $\pi^{\delta}$ in the following sens

\begin{center}
$\mathcal{H}\left(P\right)$ \\
$\equiv$\\
For every $\delta,\overline{\delta}>0, t,s \in \pi^{\delta} \cap \pi^{\overline{\delta}}$, we have $P^{\delta}_{s,t}=P^{\overline{\delta}}_{s,t}=:P_{s,t}$. 
\end{center}

From now, even if not stated, $P^{\delta}$ will be supposed to satisfy $\mathcal{H}\left(P\right)$ unless it is said otherwise so we simply denote $P$. \\
\begin{remark}
A typical example is to consider a Markov process $(X_{t})_{t \geqslant 0}$ and define $P_{s,t}f(x)=\mathbb{E}[f(X_t) \vert X_s=x]$ for every measurable and bounded function $f$, and every $0 \leqslant s \leqslant t$, $x \in \mathbb{R}^d$. 
\end{remark}

In a first step, we suppose that the semigroup we study is such that for every $r\in \mathbb{N}$, if $f\in
\mathcal{C}_b^{r}(\mathbb{R}^{d})$ then $P_{s,t}f\in \mathcal{C}_b^{r}(\mathbb{R}^{d})$ and%
\begin{align}
\sup_{t \geqslant s \geqslant 0}\Vert P_{s,t}f\Vert _{r,\infty }\leqslant C\Vert f\Vert _{r,\infty }.  
\label{hyp:transport_regularite_semigroup} 
\end{align}
Notice that (\ref{eq:borne_sup_module_mesure_VT}) implies that (\ref{hyp:transport_regularite_semigroup}) holds for $r=0$.
In this article, we expose a method to build an approximation of $P_{T}$ where $P$ satisfies $\mathcal{H}\left(P\right)$. The approximation we use is built from a family of discrete semigroups $Q:=\left( Q^{\delta_{n}^{l}} \right)_{l \in \mathbb{N}}=\left( ( Q^{\delta_{n}^{l}}_{s,t} )_{s,t \in \pi^{\delta_{n}^{l}};s \leqslant t} \right)_{l \in \mathbb{N}}$ such that for every $f\in
\mathcal{C}_b^{r}(\mathbb{R}^{d})$ then $Q^{\delta_{n}^{l}}_{s,t}f\in \mathcal{C}_b^{r}(\mathbb{R}^{d})$ and
\begin{align}
 \forall s,t \in \pi^{\delta_{n}^{l}}, s \leqslant t ,\quad  \Vert Q^{\delta_{n}^{l}}_{s,t}f\Vert
_{r,\infty}\leqslant C\Vert f\Vert _{r,\infty}. 
\label{hyp:transport_regularite_semigroup_approx} 
\end{align}

and which satisfy the short-time estimate, for every $r \in \mathbb{N}$, and $f\in
\mathcal{C}_b^{\beta+r}(\mathbb{R}^{d})$
\begin{align}
E_n(l,\alpha,\beta,P,Q)  \qquad  \forall t \in \pi^{\delta_{n}^{l}} , \quad   \Vert P_{t,t+\delta_{n}^{l}}f-Q^{\delta_{n}^{l}}_{t,t+\delta_{n}^{l}}f \Vert_{\infty,r}   \leqslant C \Vert f\Vert _{\infty,\beta+r} \left( \delta_{n}^{l} \right)^{\alpha+1}.  \label{hyp:erreur_tems_cours_fonction_test_reg}
\end{align}

\subsection{Approximation results}
\paragraph{Arbitrary order weak approximation.}
Using the family $\left( ( Q^{\delta_{n}^{l}}_{s,t} )_{s,t \in \pi^{\delta_{n}^{l}};s \leqslant t} \right)_{l \in \mathbb{N}}$, for every $(l,\nu) \in \mathbb{N}^2$, we are going to build $(\hat{Q}^{\nu,\delta_{n}^{l}}_{t,t+\delta_{n}^{l}})_{t \in \pi^{\delta_{n}^{l}}}$ as an approximation of $(P_{t,t+\delta_{n}^{l}})_{t \in \pi^{\delta_{n}^{l}}}$ which, under the hypothesis (\ref{hyp:transport_regularite_semigroup}), (\ref{hyp:transport_regularite_semigroup_approx}) and (\ref{hyp:erreur_tems_cours_fonction_test_reg}), satisfies, for every $r \in \mathbb{N}$, and $f\in
\mathcal{C}_b^{\kappa(l,\nu)+r}(\mathbb{R}^{d})$
\begin{align}
\hat{E}_n(l,\nu,\kappa,P,Q)  \qquad  \forall t \in \pi^{\delta_{n}^{l}} , \quad   \Vert P_{t,t+\delta_{n}^{l}}f-\hat{Q}^{\nu,\delta_{n}^{l}}_{t,t+\delta_{n}^{l}}f \Vert_{\infty,r}   \leqslant C \Vert f\Vert _{\infty,\kappa(l,\nu)+r} \frac{1}{n^{\nu}},
\label{hyp:erreur_tems_cours_fonction_test_regschema_nu} 
\end{align}
with 
\begin{align*}
\kappa(l,\nu)= \max \left\{ \beta m(l,v), \max_{i=1}^{m(l,v)-1} \left\{ i \kappa(l+1,q_{i}(l,v)) \right\} \right\}
\end{align*}
and
\begin{align*}
m\left( l,\nu \right) =& \left \lceil  \frac{\nu}{(1+\alpha)l + \alpha } \right \rceil \\
q_i\left( l,\nu \right) =& \nu + \lceil  i - (1+\alpha)(l+1)(i-1)  \rceil, \quad \forall i \in \{1,\ldots,m\left( l,\nu \right) - 1\}  .
\end{align*}
 In particular $\hat{Q}^{\nu,\delta_{n}^{0}}_{0,\delta_{n}^{0}}$ is an approximation of $P_{T}$ with accuracy $1/n^{\nu}$. The approach we use was first introduced in \cite{Alfonsi_Bally_2019}. Among other, it was showed in this paper that, combined with a random grid approach, the accuracy $1/n^{\nu}$ can be reached with complexity - in terms of the number of simulations of random variables with law given by a semigroup $Q^{\delta_{n}^{l}}$ for some $l \in \mathbb{N}$ - of order $n$.
 

For every $t \in \pi^{\delta_{n}^{l}}$, we define $\hat{Q}^{\nu,\delta_{n}^{l}}$ in the following recursive way 

\begin{align}
\label{eq:schema_ordre_quelconque}
\hat{Q}^{\nu,\delta_{n}^{l}}_{t,t+\delta_{n}^{l}}=& Q^{\delta_{n}^{l+1}}_{t,t+\delta_{n}^{l}}+\sum_{i=1}^{m\left( l,\nu \right)-1} \hat{I}^{\delta_{n}^{l+1}}_{t,t+\delta_{n}^{l},i} ,
\end{align}
with 
\begin{align*}
\hat{I}^{\delta_{n}^{l+1}}_{t,t+\delta_{n}^{l},i} =&   \sum_{t=t_0 < t_1<\ldots<t_{i}\leqslant t+\delta_{n}^{l} \in \pi^{\delta_{n}^{l+1}}}  \prod_{j=1}^{i} Q^{\delta_{n}^{l+1}}_{t_{j-1},t_{j}-\delta_{n}^{l+1}}   \left( \left( \hat{Q}^{q_i\left( l,\nu \right),\delta_{n}^{l+1}}_{t_{j}-\delta_{n}^{l+1},t_{j}} - Q^{\delta_{n}^{l+1}}_{t_{j}-\delta_{n}^{l+1},t_{j}}   \right) \right) Q^{\delta_{n}^{l+1}}_{t_{i},t+\delta_{n}^{l}}  .
\end{align*}
Notice that the recursion ends when $m\left( l,\nu \right)=1$ and in this case $\hat{Q}^{\nu,\delta_{n}^{l}}_{t,t+\delta_{n}^{l}}= Q^{\delta_{n}^{l+1}}_{t,t+\delta_{n}^{l}}$. When $m\left( l,\nu \right)>1$, the recursion still ends for every $(l,\nu) \in \mathbb{N}^2$ and $\hat{Q}^{\nu,\delta_{n}^{l}}_{t,t+\delta_{n}^{l}}$ is well-defined. This is a direct consequence of Lemme 3.8 in \cite{Alfonsi_Bally_2019}.  We also invite the reader to refer to this article for the proof of $\hat{E}_n(l,\nu,\kappa,P,Q)$ for every $(l,\nu) \in \mathbb{N}^2$. In particular, $\hat{Q}^{\nu,\delta_{n}^{0}}_{0,\delta_{n}^{0}}$ is well defined, satisfies $\hat{E}(0,\nu,\kappa,P,Q)$ and may be built from the family $\left(( Q^{\delta_{n}^{l}}_{t})_{t \in \pi^{\delta_{n}^{l}}} \right)_{l \in \{1,\ldots,l(\nu,\alpha) \}}$ with $l(\nu,\alpha)=\lceil  \nu / \alpha  \rceil$

\paragraph{Arbitrary order total variation converge.}Our purpose is to obtain a similar estimation as $\hat{E}(0,\nu,\kappa,P,Q)$ for $\Vert P_{T}f-\hat{Q}^{\nu,\delta_{n}^{0}}_{0,T}f \Vert_{\infty}$ but which remains valid for simply bounded and measurable test functions $f$. In other words we want to show that $\hat{E}(0,\nu,0)$ holds. We will obtain such results using a dual approach. In particular, for a functional operator $Q$, we denote by $Q^{\ast}$ its dual operator for the scalar product in $\mbox{L}_2(\mathbb{R}^d)$ (\textit{i.e.} $%
\left\langle Qg,f\right\rangle_{\mbox{L}_2(\mathbb{R}^d)} =\left\langle
g,Q^{\ast}f\right\rangle_{\mbox{L}_2(\mathbb{R}^d)} $). Our approach requires to introduce some additional assumptions concerning our discrete semigroups. A first step is to consider a dual version of (\ref{hyp:transport_regularite_semigroup}) and (\ref{hyp:transport_regularite_semigroup_approx}). We assume that for every $f\in
\mathcal{C}_1^{r}(\mathbb{R}^{d})$, then $P^{\ast}_{s,t}f \in \mathcal{C}_1^{r}(\mathbb{R}^{d})$ and
\begin{align}
\sup_{t \geqslant s \geqslant 0}\Vert P^{\ast}_{s,t}f\Vert _{r,1 }\leqslant C\Vert f\Vert _{r,1}.  
\label{hyp:transport_regularite_semigroup_dual} 
\end{align}
and for the family of semigroups $Q=\left( \left( Q^{\delta_{n}^{l}}_{s,t} \right)_{s,t \in \pi^{\delta_{n}^{l}};s \leqslant t} \right)_{l \in \mathbb{N}}$,  $Q^{\delta_{n}^{l},\ast}_{s,t}f \in \mathcal{C}_1^{r}(\mathbb{R}^{d})$ and
\begin{align}
 \forall s,t \in \pi^{\delta_{n}^{l}}, s \leqslant t ,\quad  \Vert Q^{\delta_{n}^{l},\ast}_{s,t}f\Vert
_{r,1}\leqslant C\Vert f\Vert _{r,1}. 
\label{hyp:transport_regularite_semigroup_approx_dual} 
\end{align}
Moreover, we assume that the following dual estimate of the error in short time holds: for every $r \in \mathbb{N}$, and $f\in
\mathcal{C}_1^{\beta+r}(\mathbb{R}^{d})$
%

\begin{align}
E_n(l,\alpha,\beta,P,Q)^{\ast}  \qquad  \forall t \in \pi^{\delta_{n}^{l}} , \quad   \Vert P^{\ast}_{t,t+\delta_{n}^{l}}f-Q^{\delta_{n}^{l},\ast}_{t,t+\delta_{n}^{l}}f    \Vert_{r,1} \leqslant C\Vert f\Vert_{\beta+r,1}\left( \delta_{n}^{l} \right)^{\alpha+1}.  \label{hyp:erreur_tems_cours_fonction_test_reg_dual}
\end{align}

At this point, we notice that using the same approach as in \cite{Alfonsi_Bally_2019}, we can derive from (\ref{hyp:transport_regularite_semigroup_dual}), (\ref{hyp:transport_regularite_semigroup_approx_dual}) and (\ref{hyp:erreur_tems_cours_fonction_test_reg_dual}),  that for every $(l,\nu) \in \mathbb{N}^{\ast}$, and $f\in
\mathcal{C}^{\kappa(l,\nu)+r}(\mathbb{R}^{d})$
\begin{align}
\hat{E}_n(l,\nu,\kappa,P,Q)^{\ast}  \qquad  \forall t \in \pi^{\delta_{n}^{l}} , \quad   \Vert P^{\ast}_{t,t+\delta_{n}^{l}}f-\hat{Q}^{\nu,\delta_{n}^{l},\ast}_{t,t+\delta_{n}^{l}}f    \Vert_{r,1} \leqslant C \Vert f\Vert_{\kappa(l,\nu)+r,1} \frac{1}{n^{\nu}}.  \label{hyp:erreur_tems_cours_fonction_test_reg_schema_nu}
\end{align}

Now, we introduce some regularization properties that will be necessary to obtain total variation convergence. In concrete applications, the property is not necessarily satisfied by the discrete semigroup $Q^{\delta}$ but by a family of functional operators close enough to it in a sense we precise later. We call this family, a modification of $Q^{\delta}$ and it is not necessarily a semigroup. Hence, this hypothesis is expressed not only for discrete semigroups but for discrete family of functional operators. \\

Let $\eta:
\mathbb{R}_+ \to \mathbb{R}_+$ an increasing function, $ \; q \in \mathbb{N}$ be fixed. For $\delta>0$, let $\left(Q ^{\delta}_{s,t}\right)_{s,t \in \pi^{\delta} ,t>s}$, be a family of functional operators. We consider the following regularization property 

\begin{center}
$R_{q,\eta }(Q^{\delta})$ \\
$\equiv$
\end{center}

  \begin{center}
    \begin{minipage}{0.8\textwidth}
\begin{enumerate}[label=$R_{q,\eta }(Q^{\delta})$\textbf{.\roman*.}]
\item \label{reg_property_mes_cont} For every $f \in \mathcal{M}_b(\mathbb{R}^d)$ and every $s,t \in \pi^{\delta_{n}^{l}} ,t>s$, $Q_{s,t}f \in \mathcal{C}_b^{\infty}(\mathbb{R}^d)$.\\
\item \label{reg_property} For every $r \in \mathbb{N}$ and every multi-index $\gamma $ with $\vert \gamma \vert
+r \leqslant q$, and $f \in \mathcal{C}_c^{\infty}(\mathbb{R}^{d})$, then
\begin{align} \qquad  \forall s,t \in \pi^{\delta} ,t>s, \quad  \Vert Q^{\delta}_{s,t} \partial _{\gamma}f\Vert
_{r,\infty}\leqslant \frac{C}{(t-s)^{\eta (q)}}\Vert f\Vert _{\infty}.  
\label{hyp:reg_forte}
\end{align}
\end{enumerate}
\end{minipage}
\end{center}

In our approach we will not use directly \ref{reg_property} but an estimate it implies on the adjoint semigroup. Actually, remembering that $\Vert f \Vert_1 \leqslant \sup_{g \in C_c^{\infty}, \Vert g \Vert_{\infty}=1} \langle g,f \rangle$, we notice that \ref{reg_property} implies that for every $r \in \mathbb{N}$ and every multi-index $\gamma $ with $\vert \gamma \vert
+r \leqslant q$ and $f\in
\mathcal{C}_c^{\infty}(\mathbb{R}^{d})$
\begin{align}
\label{hyp:reg_forte_dual}
\forall t \in \pi^{\delta_{n}^{l}} ,t>s, \quad  \Vert Q^{\delta,\ast}_{s,t} \partial _{\gamma}f\Vert
_{r,1}\leqslant \frac{C}{(t-s)^{\eta (q)}}\Vert f\Vert _{1}. 
\end{align}

Using those hypothesis, we can derive the following total variation convergence towards the semigroup $P$ with rate $1/n^{\nu}$ with $\nu$ choosen arbitrarily in $\mathbb{N}$.

\begin{theorem}
\label{th:intro_erreur_faible}
We recall that $T>0$ and $n \in \mathbb{N}^{\ast}$. Let $\nu \in \mathbb{N}$ and $\eta: \mathbb{R}_+ \to \mathbb{R}_+$ be an increasing function and define $q_{\nu}=\max_{i\in \{1,\ldots,m(0,\nu) \}}(i\max(\beta,\kappa(1,q_i(\nu,0)))$. \\
 
 Assume that (\ref{hyp:transport_regularite_semigroup}), (\ref{hyp:transport_regularite_semigroup_approx}), (\ref{hyp:transport_regularite_semigroup_dual}) and (\ref{hyp:transport_regularite_semigroup_approx_dual}) hold and that the short time estimates $E_n(l,\alpha,\beta,P,Q)$ (see (\ref{hyp:erreur_tems_cours_fonction_test_reg})), and $E_n(l,\alpha,\beta,P,Q)^{\ast}$ (see (\ref{hyp:erreur_tems_cours_fonction_test_reg_dual})) hold for every $l \in \{1,\ldots,l(\nu,\alpha)\}$. Moreover, assume that $R_{q_{\nu},\eta }(Q^{\delta_{n}^{1}})$ and $R_{q_{\nu},\eta}(P)$ (see \ref{hyp:reg_forte}) hold. Then, for every $f \in \mathcal{M}_b(\mathbb{R}^d)$,
\begin{align*}
\Vert P_{T}f-\hat{Q}^{\nu,\delta_{n}^{0}}_{0,T}f \Vert_{\infty }\leqslant \frac{C}{T(\nu)^{\eta (q_{\nu})}} \Vert f \Vert _{\infty }  \frac{1}{n^{\nu}}.  
\end{align*}
with $T(\nu)=\inf \left\{t \in \pi^{\delta_{n}^{1}},t \geqslant T\frac{n-m(0,v)}{n(m(0,v)+1)} \right\}$.
\end{theorem}

\begin{proof} 
In order to prove this result, we introduce a reprensentation for the semigroup $(P_{t})_{t \geqslant 0}$ which relies on the family of semigroup $\left( \left( Q^{\delta_{n}^{l}}_{t} \right)_{t \in \pi^{\delta_{n}^{l}}} \right)_{l \in \mathbb{N}}$. In particular we have, for every $t \in\pi^{\delta_{n}^{l}}$,
\begin{align*}
P_{t,t+\delta_{n}^{l}}=Q^{\delta_{n}^{l+1}}_{t,t+\delta_{n}^{l}}+\sum_{i=1}^{m\left( l,\nu \right)-1} I^{\delta_{n}^{l+1}}_{t,t+\delta_{n}^{l},i}+ R^{\delta_{n}^{l+1}}_{t,t+\delta_{n}^{l},m\left( l,\nu \right)}
\end{align*}
with 
\begin{align*}
I^{\delta_{n}^{l+1}}_{t,t+\delta_{n}^{l},i} =& \sum_{t=t_0 <\ldots<t_{i}\leqslant t+\delta_{n}^{l} \in \pi^{\delta_{n}^{l+1}}} \prod_{j=1}^{i} \left( Q^{\delta_{n}^{l+1}}_{t_{j-1},t_{j}-\delta_{n}^{l+1}}  \left( P_{t_{j}-\delta_{n}^{l+1},t_{j}} - Q^{\delta_{n}^{l+1}}_{t_{j}-\delta_{n}^{l+1},t_{j}}   \right)   \right)  Q^{\delta_{n}^{l+1}}_{t_{i},t+\delta_{n}^{l}} .
\end{align*}
and
\begin{align*}
R^{\delta_{n}^{l+1}}_{t,t+\delta_{n}^{l},m} =& \sum_{t=t_0 <\ldots<t_{m}\leqslant t+\delta_{n}^{l} \in \pi^{\delta_{n}^{l+1}}}  \prod_{j=1}^{m} \left( Q^{\delta_{n}^{l+1}}_{t_{j-1},t_{j}-\delta_{n}^{l+1}} \left( P_{t_{j}-\delta_{n}^{l+1},t_{j}} - Q^{\delta_{n}^{l+1}}_{t_{j}-\delta_{n}^{l+1},t_{j}}   \right)   \right)  P_{t_{m},t+\delta_{n}^{l}}.
\end{align*}

It follows that
\begin{align*}
P_{T}f-\hat{Q}^{\nu,\delta_{n}^{0}}_{0,T}= \sum_{i=1}^{m\left(0,\nu \right)-1} I^{\delta_{n}^{1}}_{0,T,i}-\hat{I}^{\delta_{n}^{1}}_{0,T,i}+ R^{\delta_{n}^{1}}_{0,T,m\left(0,\nu \right)}
\end{align*}
with
\begin{align*}
I^{\delta_{n}^{1}}_{0,T,i}-\hat{I}^{\delta_{n}^{1}}_{0,T,i}=&  \sum_{0=t_0 <\ldots<t_{i}\leqslant T \in \pi^{\delta_{n}^{1}}}  \prod_{j=1}^{i} \left( Q^{\delta_{n}^{1}}_{t_{j-1},t_{j}-\delta_{n}^{1}} \left( P_{t_{j}-\delta_{n}^{1},t_{j}}  -\hat{Q}^{q_i\left( l,\nu \right),\delta_{n}^{1}}_{t_{j}-\delta_{n}^{1},t_{j}} +Q^{\delta_{n}^{1}}_{t_{j}-\delta_{n}^{1},t_{j}}  - P_{t_{j}-\delta_{n}^{1},t_{j}}  \right)   \right) Q^{\delta_{n}^{1}}_{t_{i},T} \\
& -\sum_{0=t_0 <\ldots<t_{i}\leqslant T \in \pi^{\delta_{n}^{1}}}  \prod_{j=1}^{i} \left( Q^{\delta_{n}^{1}}_{t_{j-1},t_{j}-\delta_{n}^{1}}  \left( P_{t_{j}-\delta_{n}^{1},t_{j}} - Q^{\delta_{n}^{1}}_{t_{j}-\delta_{n}^{1},t_{j}}   \right)  \right) Q^{\delta_{n}^{1}}_{t_{i},T} \\
\end{align*}
More particularly, we can write
\begin{align*}
I^{\delta_{n}^{1}}_{0,T,i}-\hat{I}^{\delta_{n}^{1}}_{0,T,i}=&\sum_{0=t_0 <\ldots<t_{i}\leqslant T \in \pi^{\delta_{n}^{1}}} \sum_{h=1}^{2^i} \left(  \prod_{j=1}^{i} Q^{\delta_{n}^{1}}_{t_{j-1},t_{j}-\delta_{n}^{1}} \Lambda^{\delta_{n}^{1},h}_{t_{j}-\delta_{n}^{1},t_{j}}    \right)  Q^{\delta_{n}^{1}}_{t_{i},T}
\end{align*}
with, for every $j \in \{1,\ldots,i\}$, $ \Lambda^{\delta_{n}^{1},h}_{t_{j}-\delta_{n}^{1},t_{j}} \in \left\{ Q^{\delta_{n}^{1}}_{t_{j}-\delta_{n}^{1},t_{j}}-P_{t_{j}-\delta_{n}^{1},t_{j}}, P_{t_{j}-\delta_{n}^{1},t_{j}}  -\hat{Q}^{q_i\left( 0,\nu \right),\delta_{n}^{1}}_{t_{j}-\delta_{n}^{1},t_{j}} \right\}$. Moreover, we notice that the case $\Lambda^{\delta_{n}^{1},h}_{t_{j}-\delta_{n}^{1},t_{j}}= Q^{\delta_{n}^{1}}_{t_{j}-\delta_{n}^{1},t_{j}}-P_{t_{j}-\delta_{n}^{1},t_{j}}$ for every $j \in \{1,\ldots,i\}$ is excluded.
Using this decomposition, it is is sufficient to prove that 
\begin{align*}
\left \Vert  \left(  \prod_{j=1}^{i}  Q^{\delta_{n}^{1}}_{t_{j-1},t_{j}-\delta_{n}^{1}}  \Lambda^{\delta_{n}^{1},h}_{t_{j}-\delta_{n}^{1},t_{j}}   \right) Q^{\delta_{n}^{1}}_{t_{i},T}  f \right \Vert _{\infty
}\leqslant \frac{C}{%
T(\nu)^{\eta (q)}}\Vert f\Vert _{\infty } / n^{\nu+i},
\end{align*}
and, for the remainder,
\begin{align*}
\left \Vert   \left( \prod_{j=1}^{m\left( l,\nu \right)} Q^{\delta_{n}^{1}}_{t_{j-1},t_{j}-\delta_{n}^{1}} \ \left( P_{t_{j}-\delta_{n}^{1},t_{j}} - Q^{\delta_{n}^{1}}_{t_{j}-\delta_{n}^{1},t_{j}}   \right) \right)P_{t_{m\left( l,\nu \right)},T}  f \right \Vert _{\infty
}\leqslant \frac{C}{%
T(\nu)^{\eta (q)}}\Vert f \Vert _{\infty } / n^{\nu+i}.
\end{align*}
We study the first term to estimate. The study of the remainder is similar so we leave it out. First we notice that, using the convention $t_{i+1}=T+\delta_{n}^{1}$, for ${j_i}=\argsup_{j\in \{1,\ldots,i+1\}} \{ t_{j}-\delta_{n}^{1}-t_{j-1} \}$, we have $t_{j_i}-\delta_{n}^{1}-t_{j_i-1}\geqslant T(\nu)$. Using succesively $E_n(1,\alpha,\beta,P,Q)$ (see (\ref{hyp:erreur_tems_cours_fonction_test_reg})) or  $\hat{E}_n(1,q_i(\nu,0),\kappa)$ (see (\ref{hyp:erreur_tems_cours_fonction_test_regschema_nu})) with (\ref{hyp:transport_regularite_semigroup_approx}) -which can be applied since $R_{q_{\nu},\eta }(Q^{\delta_{n}^{1}})\mathbf{.i.}$ holds -, it follows that


\begin{align*}
& \left \Vert  \left( \prod_{j=1}^{i}  Q^{\delta_{n}^{1}}_{t_{j-1},t_{j}-\delta_{n}^{1}} \Lambda^{\delta_{n}^{1},h}_{t_{j}-\delta_{n}^{1},t_{j}} \right)    Q^{\delta_{n}^{1}}_{t_{i},T} f \right \Vert _{\infty}\\
&= \left \Vert  \left( \prod_{j=1}^{j_i-1}   Q^{\delta_{n}^{1}}_{t_{j-1},t_{j}-\delta_{n}^{1}} \Lambda^{\delta_{n}^{1},h}_{t_{j}-\delta_{n}^{1},t_{j}} \right) \left(  Q^{\delta_{n}^{1}}_{t_{j_i-1},t_{j_i}-\delta_{n}^{1}}   \Lambda^{\delta_{n}^{1},h}_{t_{j_i}-\delta_{n}^{1},t_{j_i}} \right) \left(    \prod_{j=j_i+1}^{i} Q^{\delta_{n}^{1}}_{t_{j-1},t_{j}-\delta_{n}^{1}} \Lambda^{\delta_{n}^{1},h}_{t_{j}-\delta_{n}^{1},t_{j}} \right)    Q^{\delta_{n}^{1}}_{t_{i},T}   f \right \Vert _{\infty
} \\ 
&\leqslant  C \left \Vert    \left(  Q^{\delta_{n}^{1}}_{t_{j_i-1},t_{j_i}-\delta_{n}^{1}}   \Lambda^{\delta_{n}^{1},h}_{t_{j_i}-\delta_{n}^{1},t_{j_i}} \right) \left(    \prod_{j=j_i+1}^{i} Q^{\delta_{n}^{1}}_{t_{j-1},t_{j}-\delta_{n}^{1}} \Lambda^{\delta_{n}^{1},h}_{t_{j}-\delta_{n}^{1},t_{j}} \right)    Q^{\delta_{n}^{1}}_{t_{i},T}    f \right \Vert _{(j_i-1) \max(\beta,\kappa(1,q_i(\nu,0))),\infty } \frac{1}{n^{\sum_{j=1}^{j_i-1} q_i^h(j)}},
\end{align*}

with $q_i^h(j)=q_i(\nu,0)$ if $ \Lambda^{\delta_{n}^{1},h}_{t_{j}-\delta_{n}^{1},t_{j}} =P_{t_{j}-\delta_{n}^{1},t_{j}}  -\hat{Q}^{q_i\left( 0,\nu \right),\delta_{n}^{1}}_{t_{j}-\delta_{n}^{1},t_{j}}$ and  $q_i^h(j)=\alpha$ if $ \Lambda^{\delta_{n}^{1},h}_{t_{j}-\delta_{n}^{1},t_{j}} = Q^{\delta_{n}^{1}}_{t_{j}-\delta_{n}^{1},t_{j}}-P_{t_{j}-\delta_{n}^{1},t_{j}}$. Notice that it is not possible to have $q_i^h(j) = \alpha$ for every $j \in \{1,\ldots,i\}$ and that there exists $\hat{j}_i \in \{1,\ldots,i\}$, such that
\begin{align*}
\sum_{j=1}^{i} q_i^h(j)=\hat{j}_i q_i(\nu,0)+(i-\hat{j}_i)\alpha \geqslant \nu + i,
\end{align*}
Now, for $\varepsilon > 0$, we consider $\phi _{\varepsilon
}(x)=\varepsilon ^{-d}\phi (\varepsilon ^{-1}x)$ with $\phi \in
\mathcal{C}^{\infty}_{c}(\mathbb{R}^{d}),\; \phi \geqslant 0.$ Then, for a fixed $x_{0},$ we define $\phi
_{\varepsilon ,x_{0}}(x)=\phi _{\varepsilon }(x-x_{0}).$ Moreover, denote 
\begin{align*}
\Gamma_i=\Lambda^{\delta_{n}^{1},h}_{t_{j_i}-\delta_{n}^{1},t_{j_i}}  \left(    \prod_{j=j_i+1}^{i} Q^{\delta_{n}^{1}}_{t_{j-1},t_{j}-\delta_{n}^{1}} \Lambda^{\delta_{n}^{1},h}_{t_{j}-\delta_{n}^{1},t_{j}} \right)    Q^{\delta_{n}^{1}}_{t_{i},T}   
\end{align*}

 Since we have (\ref{hyp:transport_regularite_semigroup}) and (\ref{hyp:transport_regularite_semigroup_approx}),
$ Q^{\delta_{n}^{1}}_{t_{j-1},t_{j}-\delta_{n}^{1}}\Gamma_i   f $ belongs to $\mathcal{C}^{\infty}$. Using succesively $E_n(1,\alpha,\beta,P,Q)^{\ast}$ (see (\ref{hyp:erreur_tems_cours_fonction_test_reg_dual})) or  $\hat{E}(1,q_i(\nu,0),\kappa)^{\ast}$ (see (\ref{hyp:erreur_tems_cours_fonction_test_reg_schema_nu})) with (\ref{hyp:transport_regularite_semigroup_approx_dual}), it follows that for a multi-index $\gamma$, $x_0 \in \mathbb{R}^d$,

\begin{align*}
\vert  \partial^{\gamma}_x  Q^{\delta_{n}^{1}}_{t_{j-1},t_{j}-\delta_{n}^{1}}\Gamma_i    & f (x_{0})\vert =  \lim_{\varepsilon
\rightarrow 0}\vert \langle  \Gamma_i^{\ast} Q^{\delta_{n}^{1},\ast}_{t_{j-1},t_{j}-\delta_{n}^{1}} \partial^{\gamma}_x \phi_{ \varepsilon,x_{0}}, f \rangle \vert \\
\leqslant & C \sup_{\varepsilon >0} \left \Vert Q^{\delta_{n}^{1},\ast}_{t_{j_i-1},t_{j_i}-\delta_{n}^{1}} \partial^{\gamma}_x \phi_{ \varepsilon,x_{0}} \right\Vert _{ (i-j_i+1) \max(\beta,\kappa(1,q_i(\nu,0))),1}  \left\Vert f \right\Vert_{\infty}  \frac{1}{n^{\sum_{i=1}^{j_i-1} q_i^h(j)}}  \\
\end{align*}%
Our concern is the case $\vert \gamma \vert \leqslant (j_i-1) \max(\beta,\kappa(1,q_i(\nu,0)))$. Using $R_{q_{\nu},\eta }(Q^{\delta_{n}^{1}})\mathbf{.ii.}$ (see (\ref{hyp:reg_forte})) and more particularly the implication (\ref{hyp:reg_forte_dual}) on $Q^{\delta_{n}^{1},\ast}$, it follows that,

\begin{align*}
\left \Vert  \left( \prod_{j=1}^{i}  Q^{\delta_{n}^{1}}_{t_{j-1},t_{j}-\delta_{n}^{1}} \Lambda^{\delta_{n}^{1},h}_{t_{j}-\delta_{n}^{1},t_{j}} \right)    Q^{\delta_{n}^{1}}_{t_{i},T} f \right \Vert _{\infty} \leqslant  C \sup_{\varepsilon >0} \frac{1}{T(\nu)^{\eta(i \max(\beta,\kappa(1,q_i(\nu,0))))}} \left\Vert \phi _{\varepsilon ,x_{0}} \right\Vert _{1}  \left\Vert f \right\Vert_{\infty}  \frac{1}{n^{\sum_{j=1}^{i} q_i^h(j)}}  
\end{align*}

and since $\sum_{j=1}^{i} q_i^h(j) \geqslant \nu + i$ and $\Vert \phi _{\varepsilon ,x_{0}}\Vert
_{1}=\Vert \phi \Vert _{1}\leqslant C,$\ the proof is completed.

\end{proof}

We are now interested by giving a variant of Theorem in which rthe regularization hypothesis is not required for $P$ or $Q$ but for some modifications of those simegroups.
\begin{proposition}
\label{prop:intro_erreur_faible_modif}
 We recall that $T>0$ and $n \in \mathbb{N}^{\ast}$. Let $q_{\nu}=\max_{i\in \{1,\ldots,m(0,\nu) \}}(i\max(\beta,\kappa(1,q_i(\nu,0)))$. We assume that (\ref{hyp:transport_regularite_semigroup}) and (\ref{hyp:transport_regularite_semigroup_approx}) and (\ref{hyp:transport_regularite_semigroup_dual}), (\ref{hyp:transport_regularite_semigroup_approx_dual}) hold and that the short time estimates $E_n(l,\alpha,\beta,P,Q)$ (see (\ref{hyp:erreur_tems_cours_fonction_test_reg})), and $E_n(l,\alpha,\beta,P,Q)^{\ast}$ (see (\ref{hyp:erreur_tems_cours_fonction_test_reg_dual})) hold for every $l \in \{1,\ldots,l(\nu,\alpha)\}$. Also, assume that there exists a modification $\overline{Q}^{\delta_{n}^{1}}$ (respectively $\overline{P}$) of $Q^{\delta_{n}^{1}}$ (resp. $P$) which satisfy  $R_{q_{\nu},\eta }(\overline{Q}^{\delta_{n}^{1}})$ (resp. $R_{q_{\nu},\eta }(\overline{P})$) (see \ref{hyp:reg_forte}) and such that for every $f \in \mathcal{M}_b(\mathbb{R}^d)$,
 
 \begin{align}
\forall s,t \in \pi^{\delta_{n}^{1}}, s < t  , \qquad \Vert Q^{\delta_{n}^{1}}_{s,t}f- \overline{Q}^{\delta_{n}^{1}}_{s,t}f\Vert _{\infty }+\Vert
P_{s,t}f-\overline{P}_{s,t}f\Vert _{\infty }\leqslant C(t-s)^{-\eta (q_{\nu})}\Vert f\Vert _{\infty }/n^{\nu+m(0,\nu)}.  
 \label{hyp:approx_distance_total_variation_loc}
\end{align}

 Then, for every $f \in \mathcal{M}_b(\mathbb{R}^d)$,
\begin{align*}
\Vert P_{T}f-\hat{Q}^{\nu,\delta_{n}^{0}}_{0,T}f \Vert_{\infty }\leqslant \frac{C}{T(\nu)^{\eta (q_{\nu})}} \Vert f \Vert _{\infty }  \frac{1}{n^{\nu}}.  
\end{align*}
with $T(\nu)=\inf \left\{t \in \pi^{\delta_{n}^{1}},t \geqslant T\frac{n-m(0,v)}{n(m(0,v)+1)} \right\}$.
\end{proposition}

\begin{remark}
Notice that $\overline{P}$ and $\overline{Q}^{\delta_{n}^{1}}$ are not supposed
to satisfy the semigroup property and are not directly related to $\mu$
and $\nu$
\end{remark}

\begin{proof}
 The proof follows the same line as the one of the previous
Theorem \ref{th:intro_erreur_faible}. Consequently, we only focus on the specificity of this proof, avoiding arguments which are similar to the previous proof. In particular we study, for every $i \in \{1,\ldots,m(0,\nu)\}$,

\begin{align*}
\Vert  Q^{\delta_{n}^{1}}_{t_{i},T} & \prod_{j=1}^{i} \Lambda^{\delta_{n}^{1},h}_{t_{j}-\delta_{n}^{1},t_{j}}   Q^{\delta_{n}^{1}}_{t_{j-1},t_{j}-\delta_{n}^{1}}  f\Vert _{\infty
}\\
\leqslant &  \Vert  Q^{\delta_{n}^{1}}_{t_{i},T}   \prod_{j=j_i+1}^{i} \Lambda^{\delta_{n}^{1},h}_{t_{j}-\delta_{n}^{1},t_{j}}   Q^{\delta_{n}^{1}}_{t_{j-1},t_{j}-\delta_{n}^{1}} \Lambda^{\delta_{n}^{1},h}_{t_{j_i}-\delta_{n}^{1},t_{j_i}}  \overline{Q}^{\delta_{n}^{1}}_{t_{j_i-1},t_{j_i}-\delta_{n}^{1}} \prod_{j=1}^{j_i-1} \Lambda^{\delta_{n}^{1},h}_{t_{j}-\delta_{n}^{1},t_{j}}   Q^{\delta_{n}^{1}}_{t_{j-1},t_{j}-\delta_{n}^{1}}f\Vert _{\infty
} \\ 
& +   \Vert  Q^{\delta_{n}^{1}}_{t_{i},T}   \prod_{j=j_i+1}^{i} \Lambda^{\delta_{n}^{1},h}_{t_{j}-\delta_{n}^{1},t_{j}}   Q^{\delta_{n}^{1}}_{t_{j-1},t_{j}-\delta_{n}^{1}} \Lambda^{\delta_{n}^{1},h}_{t_{j_i}-\delta_{n}^{1},t_{j_i}} \left(Q^{\delta_{n}^{1}}_{t_{j_i-1},t_{j_i}-\delta_{n}^{1}}-  \overline{Q}^{\delta_{n}^{1}}_{t_{j_i-1},t_{j_i}-\delta_{n}^{1}} \right) \prod_{j=1}^{j_i-1} \Lambda^{\delta_{n}^{1},h}_{t_{j}-\delta_{n}^{1},t_{j}}   Q^{\delta_{n}^{1}}_{t_{j-1},t_{j}-\delta_{n}^{1}}f\Vert _{\infty
} .
\end{align*}
The first term is studied similarly as in Theorem \ref{th:intro_erreur_faible} We use the same notations as introduced in this proof. For the second term we use (\ref{hyp:approx_distance_total_variation_loc}) together with successive application of (\ref{hyp:transport_regularite_semigroup_approx}) and it follows that

\begin{align*}
\Vert  Q^{\delta_{n}^{1}}_{t_{i},T}  \prod_{j=1}^{i} \Lambda^{\delta_{n}^{1},h}_{t_{j}-\delta_{n}^{1},t_{j}}   Q^{\delta_{n}^{1}}_{t_{j-1},t_{j}-\delta_{n}^{1}}  f\Vert _{\infty
} \leqslant    \frac{C}{T(\nu)^{\eta (q_{\nu})}} \Vert f \Vert _{\infty }  \frac{1}{n^{\nu+i}}+ \frac{C}{T(\nu)^{\eta (q_{\nu})}} \Vert f \Vert _{\infty }  \frac{1}{n^{\nu+m(0,\nu)}}
\end{align*}
Notice that the study of the remainder $R^{\delta_{n}^{1}}_{0,T,m\left(0,\nu \right)}$ which appears in the proof of
Proposition \ref{th:intro_erreur_faible} is similar. Rearranging the terms completes the proof.
\end{proof}

At this point, we establish a total variation convergence result which does not require that the regularization property hold for $P$ but only on the collection of semigroups $\left(Q^{\delta}\right)_{\delta>0}$. More specifically, we consider the following hypothesis : Recall that  $q_{\nu}=\max_{i\in \{1,\ldots,m(0,\nu) \}}(i\max(\beta,\kappa(1,q_i(\nu,0)))$ with the definition of $m$, $\kappa$ and $q_i$ given in \ref{hyp:erreur_tems_cours_fonction_test_regschema_nu} and that $l(\nu,\alpha)=\lceil  \nu / \alpha  \rceil$. Let us consider the hypothesis:\\

\begin{center}
$\overline{R}_{n,\nu,\eta }(Q)$ \\
$\equiv$ \\
For every $k \in \mathbb{N}^{\ast}$,
\end{center}

  \begin{center}
    \begin{minipage}{0.89\textwidth}
\begin{enumerate}[label=$\overline{R}_{n,\nu,\eta }(Q)$\textbf{.\roman*.}]
\item \label{reg_property_mes_cont_cauchy_1}  (\ref{hyp:transport_regularite_semigroup_approx}) and (\ref{hyp:transport_regularite_semigroup_approx_dual}) hold with $l$ replaced by $k$.
\item \label{reg_property_mes_cont_cauchy_3} There exists a modification $\overline{Q}^{\delta_{n}^{k}}$ of $Q^{\delta_{n}^{k}}$ which satisfies  $R_{q_{\nu},\eta }(\overline{Q}^{\delta_{n}^{k}})$  (see \ref{hyp:reg_forte}) and such that: $\forall s,t \in  \pi^{\delta_{k}^{1}}, s < t$,
 \begin{align}
\Vert Q^{\delta_{n}^{k}}_{s,t}f- \overline{Q}^{\delta_{n}^{k}}_{s,t}f\Vert _{\infty }\leqslant C(t-s)^{-\eta (q_{\nu})} \frac{\Vert f\Vert _{\infty }}{n^{\nu+m(0,\nu)}}.
 \label{hyp:approx_distance_total_variation_loc_theorem}
\end{align}
\end{enumerate}
\end{minipage}
\end{center}

\begin{theorem}
\label{theo:distance_density} 
We recall that $T>0$ and $n \in \mathbb{N}^{\ast}$. Let $\nu \in \mathbb{N}$ and $\eta: \mathbb{R}_+ \to \mathbb{R}_+$ be an increasing function.  \\
 
 Assume that (\ref{hyp:transport_regularite_semigroup}) and (\ref{hyp:transport_regularite_semigroup_dual}) hold. Assume that $\overline{R}_{n,\nu,\eta }(Q)$ hold and that for every $k \in \mathbb{N}^{\ast}$, the short time estimates $E_n(k,\alpha,\beta,P,Q)$ (see (\ref{hyp:erreur_tems_cours_fonction_test_reg})), and $E_n(k,\alpha,\beta,P,Q)^{\ast}$ (see (\ref{hyp:erreur_tems_cours_fonction_test_reg_dual})) hold. Then, for every $f \in \mathcal{M}_b(\mathbb{R}^d)$,
\begin{align}
\Vert P_{T}f-\hat{Q}^{\nu,\delta_{n}^{0}}_{0,T}f \Vert_{\infty }\leqslant \frac{C}{T(\nu)^{\eta (q_{\nu})}} \Vert f \Vert _{\infty }  \frac{1}{n^{\nu}}.  
\label{eq:theo_approx_semigroup_theorem}
\end{align}
with $T(\nu)=\inf \left\{t \in \pi^{\delta_{n}^{1}},t \geqslant T\frac{n-m(0,v)}{n(m(0,v)+1)} \right\}$.

\end{theorem}

\begin{remark}
The inequality (\ref{eq:theo_approx_semigroup_theorem}) is essentially a consequence of Theorem \ref{th:intro_erreur_faible}. However, we may not use directly this result, because we do not assume that the semigroup $P$ has the regularization property $R_{q_{\nu},\eta }(P)$ (see (\ref{hyp:reg_forte})) This is a result of main interest since we have to check the regularization properties for the approximations $Q^{\delta}$ only. Notice that the method we use does not allow to prove the same result when assuming regularization hypothesis on $P$ instead of $Q$. The reason is that our proof consist in considering $P$ as the limit of $Q^{\delta}$ as $\delta$ tends to 0. It is not possible to act similarly in the other way as $P$ does not depend on such a ${\delta}$, from hypothesis $\mathcal{H}\left(P\right)$.
\end{remark}


\begin{proof}[Proof of Theorem \ref{theo:distance_density}]

We fix $n \in \mathbb{N}^{\ast}$ and we study the sequence of discrete semigroups $\left( \left( Q^{\delta_{n}^{k}}_{s,t} \right)_{s,t \in \pi^{\delta_{n}^{k}};s \leqslant t} \right)_{k \in \mathbb{N}^{\ast}}$.

\paragraph{Step 1.} We show that for every bounded and measurable test function $f$, $\left(Q^{\delta_{n}^{k}}_{0,T}f \right)_{k \in \mathbb{N}^{\ast}}$ is Cauchy in $\mbox{L}_{\infty}$ and that
\begin{align}
\label{eq:estim_tv_cauchy_limit}
\left \Vert P_{T}f-Q^{\delta_{n}^{k}}_{0,T}f \right \Vert_{\infty} \leqslant \frac{C}{T^{\eta\left( q_{\nu} \right)}} \left \Vert f \right \Vert_{\infty} \frac{1}{n^{\nu}}.
\end{align}
 Let $k' \geqslant k \in \mathbb{N}^{\ast}$,

\begin{align*}
\left \Vert Q^{\delta_{n}^{k'}}_{0,T}f-Q^{\delta_{n}^{k}}_{0,T}f \right \Vert_{\infty} \leqslant   \sum_{m=1}^{n} \left \Vert   Q^{\delta_{n}^{k}}_{0,\left(m-1\right)\delta_{n}^{1}} \left( Q^{\delta_{n}^{k'}}_{\left(m-1\right)\delta_{n}^{1},m\delta_{n}^{1}} - Q^{\delta_{n}^{k}}_{\left(m-1\right)\delta_{n}^{1},m\delta_{n}^{1}}   \right)  Q^{\delta_{n}^{k'}}_{m\delta_{n}^{1},T}  f \right \Vert_{\infty}
\end{align*}

Now notice that for $g \in \mathcal{C}_b^{\beta}\left(\mathbb{R}^d\right)$.
\begin{align*}
\left \Vert Q^{\delta_{n}^{k'}}_{\left(m-1\right)\delta_{n}^{1},m\delta_{n}^{1}}g-Q^{\delta_{n}^{k}}_{\left(m-1\right)\delta_{n}^{1},m\delta_{n}^{1}}g \right \Vert_{\infty}  \leqslant & \left \Vert P_{\left(m-1\right)\delta_{n}^{1},m\delta_{n}^{1}}g-Q^{\delta_{n}^{k'}}_{\left(m-1\right)\delta_{n}^{1},m\delta_{n}^{1}}g \right \Vert_{\infty} \\
& +\left \Vert P_{\left(m-1\right)\delta_{n}^{1},m\delta_{n}^{1}}g-Q^{\delta_{n}^{k}}_{\left(m-1\right)\delta_{n}^{1},m\delta_{n}^{1}}g \right \Vert_{\infty}
\end{align*}
with $\left \Vert P_{\left(m-1\right)\delta_{n}^{1},m\delta_{n}^{1}}g-Q^{\delta_{n}^{k}}_{\left(m-1\right)\delta_{n}^{1},m\delta_{n}^{1}}g \right \Vert_{\infty} \leqslant C  \left \Vert g \right \Vert_{\infty,\beta} \frac{1}{n^{\alpha+1}} $ if $k=1$ (see $E_{n}(1,\alpha,\beta,P,Q)$) and if $k>1$
\begin{align*}
\big \Vert P_{\left(m-1\right)\delta_{n}^{1},m\delta_{n}^{1}}f- &   Q^{\delta_{n}^{k}}_{\left(m-1\right)\delta_{n}^{1},m\delta_{n}^{1}}g \big \Vert_{\infty} \\
 & \leqslant   \sum_{u=1+n^{k-1} \left(m-1\right)}^{mn^{k-1}} \left \Vert  Q^{\delta_{n}^{k}}_{\left(m-1\right)\delta_{n}^{1},\left(u-1\right)\delta_{n}^{k}} \left( P_{\left(u-1\right)\delta_{n}^{k},u\delta_{n}^{k}} - Q^{\delta_{n}^{k}}_{\left(u-1\right)\delta_{n}^{k},u\delta_{n}^{k}}   \right) P_{u \delta_{n}^{k},m\delta_{n}^{1}}   g \right \Vert_{\infty} \\
& \leqslant C  \left \Vert g \right \Vert_{\infty,\beta} \frac{1}{n^{k\alpha+1}}
\end{align*}

where we have used $E_{n}(k,\alpha,\beta,P,Q)$ (see (\ref{hyp:erreur_tems_cours_fonction_test_reg})). Consequently
\begin{align*}
\left \Vert Q^{\delta_{n}^{k'}}_{\left(m-1\right)\delta_{n}^{1},m\delta_{n}^{1}}g-Q^{\delta_{n}^{k}}_{\left(m-1\right)\delta_{n}^{1},m\delta_{n}^{1}}g \right \Vert_{\infty} \leqslant C  \left \Vert g \right \Vert_{\infty,\beta} \frac{1}{n^{k\alpha+1}}
\end{align*}

 In the same way we deduce from $E_{n}(k,\alpha,\beta,P,Q)^{\ast}$ (see (\ref{hyp:erreur_tems_cours_fonction_test_reg_dual})) that

\begin{align*}
\left \Vert Q^{\delta_{n}^{k'},\ast}_{\left(m-1\right)\delta_{n}^{1},m\delta_{n}^{1}}g-Q^{\delta_{n}^{k},\ast}_{\left(m-1\right)\delta_{n}^{1},m\delta_{n}^{1}}g \right \Vert_{1} \leqslant C  \left \Vert g \right \Vert_{1,\beta} \frac{1}{n^{k\alpha+1}}
\end{align*}

Combining those estimates with $R_{q_{\nu},\eta }(\overline{Q}^{\delta_{n}^{k}})$ and $R_{q_{\nu},\eta }(\overline{Q}^{\delta_{n}^{k'}})$ together with (\ref{hyp:approx_distance_total_variation_loc_theorem}), the same approach as in the proof of Proposition \ref{prop:intro_erreur_faible_modif} yiels

\begin{align}
\label{eq:estim_tv_cauchy}
\left \Vert Q^{\delta_{n}^{k'}}_{0,T}f-Q^{\delta_{n}^{k}}_{0,T}f \right \Vert_{\infty} \leqslant \frac{C}{T^{\eta\left( q_{\nu} \right)}} \left \Vert f \right \Vert_{\infty} \frac{1}{n^{k\alpha}}
\end{align}

The sequence $\left(Q^{\delta_{n}^{k}}_{0,T}f \right)_{k \in \mathbb{N}^{\ast}}$ is thus Cauchy in $\mbox{L}_{\infty}$ and then $ \lim_{k \to \infty}Q^{\delta_{n}^{k}}_{0,T}f$ exists and belong to $\mbox{L}_{\infty}$. Moreover, remember that as soon as $f \in \mathcal{C}^{\infty}_c\left(\mathbb{R}^d\right)$, (\ref{hyp:erreur_tems_cours_fonction_test_regschema_nu}) holds and then $ \lim_{k \to \infty}Q^{\delta_{n}^{k}}_{0,T}f \overset{L_{\infty}}{=}P_Tf$ so that $ \lim_{k \to \infty}Q^{\delta_{n}^{k}}_{0,T}f \overset{L_{\infty}}{=}P_Tf$ also when $f$ is bounded and measurable. Taking $k \geqslant n^{\nu/\alpha-1}$ in (\ref{eq:estim_tv_cauchy}) and letting $k'$ tends to infinity, if follows that (\ref{eq:estim_tv_cauchy_limit}) holds.
\paragraph{Step 2.} We now show that for every $k \geqslant l(\nu/\alpha)=\lceil \nu/\alpha \rceil$ and every $f \in \mathcal{M}_b$, 

\begin{align}
\label{eq;erreur_faible_cauchy}
\Vert  Q^{\delta_{n}^{k}}_{0,T}f - \hat{Q}^{\nu,\delta_{n}^{0}}_{0,T}f  \Vert_{\infty} \leqslant \frac{C}{T(\nu)^{\eta (q_{\nu})}} \Vert f \Vert _{\infty }  \frac{1}{n^{\nu}}.
\end{align}

Let $k \geqslant l(\nu,\alpha)$. We remark that if we replace  $P$ by $Q^{\delta_{n}^{k}}$, the short time estimates $E_n(l,\alpha,\beta,Q^{\delta_{n}^{k}},Q)$ (see (\ref{hyp:erreur_tems_cours_fonction_test_reg})), and $E_n(l,\alpha,\beta,Q^{\delta_{n}^{k}},Q)^{\ast}$ (see (\ref{hyp:erreur_tems_cours_fonction_test_reg_dual})) still hold for every  $l \in \{1,\ldots,l(\nu,\alpha)\}$. \\

Moreover, from \ref{reg_property_mes_cont_cauchy_3}, for every $k \in \mathbb{N}^{\ast}$, the property $R_{q_{\nu},\eta}(\overline{Q}^{\delta_{n}^{k}})$ (see \ref{hyp:reg_forte}) holds for a modification $\overline{Q}^{\delta_{n}^{k}}$ of $Q^{\delta_{n}^{k}}$ which satisfies (\ref{hyp:approx_distance_total_variation_loc_theorem}). Therefore, all the assumption of Proposition \ref{th:intro_erreur_faible} are fulfilled when we replace $P$ by $Q^{\delta_{n}^{k}}$, so that $\ref{eq;erreur_faible_cauchy}$ holds.

\paragraph{Step 3.} We combine (\ref{eq:estim_tv_cauchy_limit}) and (\ref{eq;erreur_faible_cauchy}) and it follows that for every $f \in \mathcal{M}_b$, 

\begin{align*}
\Vert  P_{T}f - \hat{Q}^{\nu,\delta_{n}^{0}}_{0,T}f  \Vert_{\infty} \leqslant \frac{C}{T(\nu)^{\eta (q_{\nu})}} \Vert f \Vert _{\infty }  \frac{1}{n^{\nu}}.
\end{align*}

\end{proof}

\paragraph{Arbitrary order approximation of density.} Not only we are interested in the total variation distance between $P$ and $Q$ but also by the distance between their density distribution when they are absolutely continuous $w.r.t.$ the Lebesgue measure. In particular we introduce a slightly more restrictive condition than \ref{reg_property}. For $\delta>0$, let $\left(Q ^{\delta}_{s,t}\right)_{s,t \in \pi^{\delta} ,t>s}$, be a family of functional operator (which is not necessarily a semigroup). We consider

\begin{center}
$R^{\mbox{a.c.}}_{q,\eta}(Q^{\delta})$ \\
$\equiv$ 
\end{center}

 \begin{center}
    \begin{minipage}{0.8\textwidth}
 \label{reg_property_abs_cont} For every $r \in \mathbb{N}$ and every multi-index $\gamma $ with $\vert \gamma \vert
+r \leqslant q+2d$, and $f \in \mathcal{C}_b^{\infty}(\mathbb{R}^{d})$, then
\begin{align} \qquad  \forall s,t \in \pi^{\delta} ,t>s, \quad  \Vert Q^{\delta}_{s,t} \partial _{\gamma}f\Vert
_{r,\infty}\leqslant \frac{C}{(t-s)^{\eta (q)}}\Vert f\Vert _{\infty}.  
\label{hyp:reg_forte_abs_cont}
\end{align}
\end{minipage}
\end{center}

%
Notice that $R^{\mbox{a.c.}}_{q,\eta}(Q^{\delta})$ implies $R_{q+2d,\eta}(Q^{\delta})\mbox{\textbf{.ii.}}$ and that if $R^{\mbox{a.c.}}_{q,\eta}(Q^{\delta})$ holds, then for all $t \in \pi^{\delta}$, $t >0$, there exists $p_{t}\in \mathcal{C}^{q}(\mathbb{R}^{d}\times \mathbb{R}^{d})$ such that $Q^{\delta}
_{t}(x,dy)=p_{t}(x,y)dy$. Moreover, for every $\vert \gamma \vert +\vert \xi \vert \leqslant q$, we have

\begin{align}
\sup_{(x,y) \in  \mathbb{R}^d \times \mathbb{R}^d}\vert
\partial _{x}^{\gamma }\partial _{y}^{\xi} p_{t}(x,y)
\vert \leqslant C t^{-\eta (q+2d)}.  \label{eq:approx_absolu_cont}
\end{align}

Indeed, let $\zeta \in \mathbb{R}^d$ and  $f_{\zeta}: \mathbb{R}^d \to \mathbb{C}, x \mapsto e^{-i \langle \zeta, x\rangle}$. Using the Fourier representation of the density function, we have
\begin{align*}
p_t(x,y)= \int_{\mathbb{R}^d} e^{i \langle \zeta, y\rangle} Q_t f_{\zeta}(x)d\zeta
\end{align*}

Now we notice that $\partial^{\xi}_y f_{\zeta}(y)=f_{\zeta}(y)(- i)^{\vert \xi \vert} \prod_{i=1}^{\vert \xi \vert} \zeta_{\xi_i}$ and it follows that for all $x,y, \in \mathbb{R}^d$,
\begin{align*}
\partial _{x}^{\alpha }\partial _{y}^{\beta }p_{t}(x,y) &=  \int_{\mathbb{R}^d} i^{\vert \beta \vert} \big( \prod_{i=1}^{\vert \xi \vert} \zeta_{\beta_i} \big) e^{i \langle \zeta, y\rangle}  \partial^{\gamma }_x (Q_t f_{\zeta})(x)d\zeta  \\
& = \int_{ [-1,1]^d} i^{\vert \xi \vert} \big( \prod_{i=1}^{\vert \xi \vert} \zeta_{\xi_i} \big) e^{i \langle \zeta, y\rangle}  \partial^{\gamma }_x (Q_t f_{\zeta})(x)d\zeta + \int_{\mathbb{R}^d \setminus  [-1,1]^d} i^{\vert \xi \vert} \big( \prod_{i=1}^{\vert \xi \vert} \zeta_{\xi_i} \big) e^{i \langle \zeta, y\rangle}  \partial^{\gamma }_x (Q_t f_{\zeta})(x)d\zeta \\
& = :  I+J 
\end{align*}

Since $\Vert f_{\zeta} \Vert_{\infty}  =1$, we use (\ref{hyp:reg_forte}) and we obtain: $\vert I \vert \leqslant C S^{-\eta ( \vert \gamma \vert) }\leqslant C S^{-\eta (q)}$. Moreover, for any multi-index $\xi'$, we have
\begin{align*}
J=  (-1)^{\vert \xi \vert} i^{\vert \xi' \vert}   \int_{\mathbb{R}^d \setminus  [-1,1]^d}  \frac{e^{i \langle \zeta, y\rangle} }{\prod_{i=1}^{\vert \xi' \vert} \zeta_{\xi'_i}} \partial_x^{\gamma } (Q_t \partial^{\xi'}  \partial^{\xi} f_{\zeta} )(x)d\zeta .
\end{align*}
We take $ \beta'  =(2,\ldots,2)$ and we obtain similarly $\vert J \vert \leqslant C S^{-\eta(q+2d)}$. We gather all the terms together and we obtain \eqref{eq:approx_absolu_cont}. \\

Using this representation we obtain the following density estimation results. A first result assumes - in the line of Theorem \ref{th:intro_erreur_faible} - the regularization hypothesis $R_{q,\eta }$ on the semigroups $Q$ and $P$. Nonetheless, a second one - in the line of Theorem \ref{theo:distance_density} - only requires regularization hypothesis on a modification of $Q$. In both cases the absolute continuity hypothesis $R^{\mbox{a.c.}}_{q,\eta }$ is not required on $P$.

\begin{theorem}
\label{theo:distance_density_function_reg} 
We recall that $T>0$ and $n \in \mathbb{N}^{\ast}$. Let $\nu \in \mathbb{N}$ and $\eta: \mathbb{R}_+ \to \mathbb{R}_+$ be an increasing function. \\

Let $q \in \mathbb{N}$. Assume one of the following hypothesis is true.

\begin{enumerate}[label=\textbf{\Alph*.}]
\item \label{point:distance_density_function_reg} In addition to hypothesis from Theorem \ref{th:intro_erreur_faible}, suppose that $R^{\mbox{a.c.}}_{q,\eta }(Q^{\delta_{n}^{l(\nu,\alpha)}})$ (see (\ref{hyp:reg_forte_abs_cont})) holds.

\item \label{point:distance_density_function} In addition to hypothesis from Theorem \ref{theo:distance_density}, suppose that $R^{\mbox{a.c.}}_{q,\eta }(\overline{Q}^{\delta^{l(\nu,\alpha)}_{n}})$ (see (\ref{hyp:reg_forte_abs_cont})) holds.
\end{enumerate}

 Then $P_{T}(x,dy)=p_{T}(x,y)dy$ and $\hat{Q}^{\nu,\delta_{n}^{0}}_{0,T}=p_{T}^{n,\nu}(x,y)dy$ with $(x,y)\mapsto p_{T}(x,y)$ and $(x,y)\mapsto p_{T}^{n,\nu}(x,y)$ belonging to $\mathcal{C}^{q}(\mathbb{R}^{d}\times \mathbb{R}^{d})$.\\\

Moreover, for every $R>0, \varepsilon, \hat{\varepsilon} \in (0,1)$, and every multi-index $\gamma$, $\xi$ with $\vert \gamma \vert + \vert \xi \vert \leqslant q$, 
\begin{align}
\sup_{(x,y) \in \overline{B}_R(x_0,y_0)}\vert \partial _{x}^{\gamma }\partial _{y}^{\xi
}p_{T}(x,y)-  \partial _{x}^{\gamma }\partial _{y}^{\xi }p_{T}^{n,\nu}(x,y)\vert \leqslant  \frac{C}{ T(\nu)^{\eta (p(\nu,\vert
\gamma \vert +\vert \xi \vert,\varepsilon,\hat{\varepsilon}) )} } \frac{1}{n^{\nu (1- \varepsilon)}} \label{eq:distance_density_function_reg}  
\end{align}

with a constant $C$ which depends on $R,x_0,y_0,T $ and on $\vert
\gamma \vert +\vert \xi \vert $ and $p(\nu,u,\varepsilon,\hat{\varepsilon}) = p_{p_{u,\hat{\varepsilon}} \vee q_{\nu}-2d,\varepsilon} \vee q_{\nu}$ with  $p_{u,\varepsilon}= (u+2d+1+2\lceil  (1- \varepsilon)(u+d)/(2\varepsilon)  \rceil).$

\end{theorem}

\begin{proof}
We prove \ref{point:distance_density_function}. The proof of \ref{point:distance_density_function_reg} follows the same line so we leave it out. First, we introduce some notations. For $p \in \mathbb{N}$, we consider the distance $d_{p}$ defined by%
\begin{align*}
d_{p}(\mu ,\nu )=\sup\big \{\vert \mathsmaller{\int} fd\mu -\mathsmaller{\int} fd\nu \Vert
:\Vert f\Vert _{p,\infty }\leqslant 1 \big\}. 
\end{align*}%
For $q, \; l\in \mathbb{N}$, $r>1$ and $f \in \mathcal{C}^p(\mathbb{R}^d \times \mathbb{R}^d)$, we denote%
\begin{align*}
\Vert f\Vert _{p,l,r}= \sum_{0 \leqslant \vert \alpha \vert \leqslant p } \big(\mathsmaller{\int} \mathsmaller{\int} (1+\vert x\vert
^{l}+\vert y\vert ^{l} )\vert  \partial_{\alpha} f(x,y)\vert
^{r}dxdy \big)^{1/r}. 
\end{align*}%

We consider the following result from\cite{Bally_Rey_2016} (see Proposition 2.9) which is a variant of Theorem 2.11 from \cite{Bally_Caramellino_2014}.

\begin{lemme}
\label{th:Bally_Caramelino_Interpol}
 Let $p, \tilde{p}\in \mathbb{N}$, $m \in \mathbb{N}^{\ast}$ and $r>1$
be given and let $r^{\ast }$ be the conjugate of $r$. We consider some finite signed measures $%
\mu (dx,dy)$ and $\mu _{g_n}(dx,dy)=g_{n}(x,y)dxdy$ with $g_n \in \mathcal{C}^{p+2m}(\mathbb{R}^d \times \mathbb{R}^d)$ and we assume that there exists $\kappa_{1},\kappa_{2} \geqslant 1$, $h >0$, such that
\begin{align}
 d_{\tilde{p}}(\mu ,\mu _{g_n}) \leqslant \kappa_{1}/n^{h} , \quad \Vert g_{n}\Vert_{p+2m,2m,r}  \leqslant \kappa_{2}, \qquad \forall n\in \mathbb{N}.  
\label{eq:borne_interpolation_suite}
\end{align}
Then $\mu(dx,dy)=g(x,y)dxdy$ where $g$ belongs to the Sobolev space $ W^{p,r}(\mathbb{R}^{d})$ and for all $\zeta>(p+\tilde{p}+d/r^{\ast })/m$, there exists a universal constant $C \geqslant 1$ such that
\begin{align}
\Vert g-g_{n}\Vert _{W^{p,r}(\mathbb{R}^{d})}\leqslant
C   \mathfrak{C}_{h,m\zeta,p+ \tilde{p}+d/ r^{\ast}} (\kappa_{2} n^{-2 h/\zeta }+ \kappa_{1} n^{-h+h(p+\tilde{p}+d/r^{\ast })/(\zeta m)}).  
\label{eq:approx_interpolation_coro}
\end{align}

with $ \mathfrak{C}_{h,\xi,u}=2^{h + u}  (1-2^ {-\xi +u})^{-1} $.
\end{lemme}

We come back to our framework. First we recall that since we have $R^{\mbox{a.c.}}_{q,\eta }(\overline{Q}^{\delta_{n}^{l(\nu,\alpha)}})$ (see (\ref{hyp:reg_forte})) then $\overline{Q}^{\delta_{n}^{l(\nu,\alpha)}}_{0,T}(x,dy)=\overline{p}^{n}_{T}(x,y)dy$ and
\begin{align}
\sup_{(x,y) \in  \mathbb{R}^d \times \mathbb{R}^d}\vert
\partial _{x}^{\gamma }\partial _{y}^{\xi} \overline{p}_{t}^{n}(x,y)
\vert \leqslant C T(\nu)^{-\eta (\vert \gamma \vert + \vert \xi \vert+2d)}.  \label{eq:approx_absolu_cont}
\end{align}

We fix $R>0$ We consider a function $%
\Phi _{R}\in \mathcal{C}_{b}^{\infty }(\mathbb{R}^{d}\times \mathbb{R}^{d})$ such that $\mathds{1}_{\overline{B}_{R}(x_0,y_0)}(x,y)\leqslant \Phi _{R}(x,y)\leqslant \mathds{1}_{B_{R+1}(x_0,y_0)}$ and we denote%
\begin{align*}
g_{T}^{n,R}(x,y)=\Phi _{R}(x,y)\overline{p}_{T}^{n}(x,y). 
\end{align*}%

We are going to use Lemma \ref{th:Bally_Caramelino_Interpol} for the sequence $g_{n}:=g_{T}^{n,R}, \;n\in \mathbb{N}$, and $\mu(dx,dy) = \Phi _{R}(x,y) \hat{Q}^{\nu,\delta_{n}^{0}}_{0,T}(x,dy)dx$. In our specific case, (\ref{eq;erreur_faible_cauchy}) with $k=l(\nu,\alpha)$ yields $d_{0}(\mu, \mu_{g_{n}})\leqslant
CT(\nu)^{\eta (q_{\nu})}n^{-\nu}$. Since we have also (\ref{eq:approx_absolu_cont}), it follows that (\ref{eq:borne_interpolation_suite}) hold with $h=\nu$, $\kappa_{1}=C T(\nu)^{-\eta (q_{\nu})}$ and $\kappa_{2}= C T(\nu)^{-\eta  (p+2m+2d) } $ where $C$ may depend on $R$. We deduce from Lemma \ref{th:Bally_Caramelino_Interpol} that  $\Phi _{R}(x,y) \hat{Q}^{\nu,\delta_{n}^{0}}_{0,T}(x,dy)dx=\mu(dx,dy)=g(x,y)dxdy$ with $g \in W^{p,r}(\mathbb{R}^{d})$. Moreover, using Sobolev's embedding theorem, for $\zeta>(p+d/r^{\ast})/m$ and $u\leqslant p-d/r$ we have
\begin{align*}
\Vert g-g_{n}\Vert _{u,\infty }\leqslant C\Vert
g-g_{n}\Vert _{W^{p,r}( \mathbb{R}^{d})}\leqslant   C  \mathfrak{C}_{\nu,m\zeta,p+d/ r^{\ast}} (T(\nu)^{-\eta  (p+2m+2d) }  n^{-2 \nu/\zeta }+ T(\nu)^{- \eta (q_{\nu}) } n^{-\nu+\nu(p+d/r^{\ast })/(\zeta m )}).  
\end{align*}%
We take $r=d$, $u=\vert \gamma \vert + \vert \xi \vert$, $p=\vert \gamma \vert + \vert \xi \vert+1$ and $m = \lceil (1- \hat{\varepsilon})(\vert \gamma \vert + \vert \xi \vert+d)/(2\hat{\varepsilon})  \rceil$ and put $\zeta=2/(1- \epsilon)$. In this case $\zeta \geqslant  (p+d/r^{\ast})/m+2$ and we obtain
\begin{align*}
\Vert g-g_{n}\Vert _{\vert \gamma \vert + \vert \xi \vert,\infty }\leqslant   C  2^{\nu +\vert \gamma \vert + \vert \xi \vert+d}  (T(\nu)^{-\eta  (p_{\vert \gamma \vert + \vert \xi \vert,\hat{\varepsilon}} )}  n^{-\nu(1- \hat{\varepsilon} ) }+ T(\nu)^{- \eta (q_{\nu}) } n^{-\nu(1-\hat{\varepsilon})}),
\end{align*}
with $p_{u,\hat{\varepsilon}}= (u+2d+1+2\lceil  (1- \hat{\varepsilon})(u+d)/(2\hat{\varepsilon})  \rceil).$ It follows that for every $R>0, \hat{\varepsilon} \in (0,1)$ and every multi-index $\gamma ,\; \xi $, we also have 

\begin{align}
\sup_{(x,y) \in \overline{B}_R(x_0,y_0)}\vert \partial _{x}^{\gamma }\partial _{y}^{\xi
}\overline{p}^{n,k}_{T}(x,y)-  \partial _{x}^{\gamma }\partial _{y}^{\xi }p^{n,\nu}_{T}(x,y)\vert \leqslant   C T(\nu)^{-\eta (p_{\vert \gamma \vert +\vert \xi \vert,\hat{\varepsilon}} \vee q_{\nu})} / n^{\nu (1- \hat{\varepsilon})} \label{eq:distance_density_interm}  
\end{align}

with a constant $C$ which depends on $R,x_0,y_0,T $ and on $\vert
\gamma \vert +\vert \xi \vert $. In particular, using (\ref{eq:approx_absolu_cont}),

\begin{align}
\sup_{(x,y) \in \overline{B}_R(x_0,y_0)} \vert \partial _{x}^{\gamma }\partial _{y}^{\xi
}p^{n,\nu}_{T}\left(x,y\right) \vert \leqslant &  C T(\nu)^{-\eta (p_{\vert
\gamma \vert +\vert \xi \vert ,\hat{\varepsilon}} \vee q_{\nu})} / n^{\nu (1- \hat{\varepsilon})}+ T(\nu)^{-\eta (\vert \gamma \vert + \vert \xi \vert+2d)} \nonumber \\
\leqslant & C T(\nu)^{-\eta (p_{\vert
\gamma \vert +\vert \xi \vert ,\hat{\varepsilon}} \vee q_{\nu})} . \label{eq:borne_distance_density_interm}  
\end{align}

Now, we use the same approach again with $g_{n}(x,y):=\Phi _{R}(x,y).p^{n,\nu}_{T}(x,y), \;n\in \mathbb{N}$ and $\mu(dx,dy) = \Phi _{R}(x,y) P_T(x,dy)dx$. It follows that, in this case, (\ref{eq:borne_interpolation_suite}) hold with $h=\nu$, $\kappa_{1}=C T(\nu)^{-\eta (q_{\nu})}$ (see (\ref{eq:theo_approx_semigroup_theorem})) and $\kappa_{2}= C T(\nu)^{-\eta  ( p_{\vert \gamma \vert +\vert \xi \vert,\hat{\varepsilon}} \vee q_{\nu} +2m)  } $ (see (\ref{eq:borne_distance_density_interm})). 

We now apply Lemma \ref{th:Bally_Caramelino_Interpol} once again and set, as before, $r=d$, $p=\vert \gamma \vert + \vert \xi \vert+1$ and $m = \lceil (1- \varepsilon)(u+d)/(2\varepsilon)  \rceil$ and put $\zeta=2/(1- \epsilon)$. In this case $\zeta \geqslant  (p+d/r^{\ast})/m+2$ and we obtain
\begin{align*}
\Vert g-g_{n}\Vert _{\vert \gamma \vert + \vert \xi \vert,\infty }\leqslant   C  2^{\nu +\vert \gamma \vert + \vert \xi \vert+d}  (T(\nu)^{-\eta  (p_{p_{\vert \gamma \vert + \vert \xi \vert,\hat{\varepsilon}} \vee q_{\nu} - 2d,\varepsilon} )}  n^{-\nu(1- \varepsilon ) }+ T(\nu)^{- \eta (q_{\nu}) } n^{-\nu(1-\varepsilon)}).
\end{align*}
It follows that for every $R>0, \varepsilon \in (0,1)$ and every multi-index $\gamma ,\; \xi $, we also have 
\begin{align*}
\sup_{(x,y) \in \overline{B}_R(x_0,y_0)}\vert \partial _{x}^{\gamma }\partial _{y}^{\xi
}p_{T}(x,y)-  \partial _{x}^{\gamma }\partial _{y}^{\xi }p_{T}^{n,\nu}(x,y)\vert \leqslant   C T(\nu)^{-\eta (p_{p_{\vert
\gamma \vert +\vert \xi \vert,\hat{\varepsilon}} \vee q_{\nu}-2d,\varepsilon} \vee q_{\nu})} / n^{\nu (1- \varepsilon)} 
\end{align*}

with a constant $C$ which depends on $R,x_0,y_0,T $ and on $\vert
\gamma \vert +\vert \xi \vert $.

\end{proof}

\section{Total variation convergence for a class of semigroups}
In this section we investigate the regularization properties of $Q^{\delta_{n}^{1}}$ and $Q^{\delta_{n}^{l(\nu,\alpha)}}$ which are crucial to derive totale variation convergence results through Theorem \ref{theo:distance_density}. In particular we propose an application where $Q^{\delta_{n}^{1}}$ and $Q^{\delta_{n}^{l(\nu,\alpha)}}$ are the discrete semigroups of discrete Markov processes defined through an abstract random recurrence. Regularization properties are then obtained for some modifications of those semigroups under ellipticity assumptions. More particularly, we will obtain regularization property for modifications of the family of discrete semigroups $(Q^{\delta})_{\delta>0}$. Our approach is similar to the one developped in \cite{Bally_Rey_2016} where regularization properties were established for such semigroups.

\subsection{A Class of Markov Semigroups}

Throughout this section, $T>0$ will still be fixed.

\paragraph{Definition of the semigroups.} For $\delta>0$, we consider two sequences of
independent random variables $Z^{\delta}_{t+\delta}\in \mathbb{R}^{N},\kappa^{\delta} _{t}\in \mathbb{R}, \; t\in
\pi^{\delta}$ and we assume that $Z^{\delta}_{t}$, $t \in \pi^{\delta}, t \in (0,T],$ are centered and verify the following property: There exists $z_{\ast}=(z_{\ast ,t})_{t \in \pi^{\delta} \cap (0,T]}$ taking its values in $\mathbb{R}^{N}$\ and $%
\varepsilon _{\ast },r_{\ast }>0$ such that for every Borel set $A\subset
\mathbb{R}^{N}$ and every $t \in \pi^{\delta} \cap (0,T],$%
\begin{equation}
L^{\delta}_{z_{\ast }}(\varepsilon _{\ast },r_{\ast })\qquad \mathbb{P}(Z^{\delta}_{t}\in A)\geqslant
\varepsilon _{\ast }\lambda (A\cap B_{r_{\ast }}(z_{\ast ,t}))  \label{hyp:lebesgue_bounded}
\end{equation}%
where $\lambda $ is the Lebesgue measure on $\mathbb{R}^{N}.$ In particular we say that the distribution of $Z^{\delta}$ is Lebesgue lower bounded. We also define
\begin{align}
M_{p}(Z^{\delta}):=1\vee \sup_{t \in (0,T] \cap \pi^{\delta}}\mathbb{E}[\vert Z^{\delta}_{t}\vert ^{p}]
\label{def:moment_inferieur_p_Z}
\end{align}
and assume that $
M_{p}(Z)<\infty $
for every $p\geqslant 1$. 

We construct the $\mathbb{R}^{d}$ valued Markov chain $(X^{\delta}_{t})_{t \in \pi^{\delta}}$ in the following way:%
\begin{equation}
X^{\delta}_{t+\delta}=\psi (\kappa _{t},X^{\delta}_{t},\sqrt{\delta} Z^{\delta}_{t+\delta}, \delta) , \quad t \in [0,T-\delta] \cap \pi^{\delta} \label{eq:schema_general}
\end{equation}%
where%
\begin{equation}
\psi \in \mathcal{C}^{\infty }(\mathbb{R} \times \mathbb{R}^{d}\times \mathbb{R}^{N} \times \mathbb{R}_+;\mathbb{R}^{d})\quad \mbox{and} \quad \psi
(\kappa ,x,0,0)=x.  \label{def:fonction_schema}
\end{equation}
We are now in a poistion to define our discrete semigroups. In particular, for $\delta>0$, for every bounded and meaurable function $f$ from $\mathbb{R}^d$ to $\mathbb{R}$, and every $x \in \mathbb{R}^d$,
\begin{align}
\forall t,s \in \pi^{\delta}, s \leqslant t, \qquad  Q^{\delta}_{s,t}f(x)=\mathbb{E}[f(X^{\delta}_{t}) \vert X^{\delta}_{s}=x] .
  \end{align}
  
We will obtain regularization properties for modifications of those discrete semigroups. Our estimates will be expressed in terms of the folllowing norm, for $ r\in \mathbb{N}^{\ast}$,

\begin{equation}
\Vert \psi \Vert _{1,r,\infty }= 1 \vee \sum_{\vert \alpha
\vert =0}^r \sum_{\vert \beta \vert +\vert \gamma \vert  
=1}^{r-\vert \alpha \vert}\Vert \partial
_{x}^{\alpha }\partial _{z}^{\beta } \partial_{t}^{\gamma}\psi \Vert_{\infty} ,
\label{eq:Norme_adhoc_fonction_schema}
\end{equation}%
and the in terms of,
\begin{equation}
\mathfrak{K}_r(\psi)=(1+\Vert \psi \Vert _{1,r,\infty })\exp
(\Vert \psi \Vert _{1,3,\infty }^{2}).  \label{eq:Constante_gronwall_fonction_schema}
\end{equation}

\paragraph{Lebesgue lower bounded distributions.} It is easy to check that (\ref{hyp:lebesgue_bounded}) holds if and only if there exists some non
negative measures $\mu _{t}$ with total mass $\mu _{t}(\mathbb{R}^{N})<1$ and a lower
semi-continuous function $\varphi \geqslant 0$ such that $\mathbb{P}(Z_{t}\in dz)=\mu
_{t}(dz)+\varphi(z-z_{\ast ,t})dz.$ Notice that the random variables $%
Z_{\delta},\ldots ,Z_{T}$ are not assumed to be identically distributed. However, the fact
that $r_{\ast }>0$ and $\varepsilon _{\ast }>0$ are the same for all $k$
represents a mild substitute of this property. In order to construct $\varphi$ we
have to introduce the following function: For $v>0$, set $\varphi _{v}:{%
\mathbb{R}^N}\rightarrow {\mathbb{R}}$ defined by
\begin{equation}
\varphi _{v}(z)=\mathds{1}_{\vert z \vert \leqslant v}+\exp \Big(1-\frac{v^{2}}{v^{2}-(\vert z \vert-v)^{2}}\Big)%
\mathds{1}_{v<\vert z \vert<2v}.  \label{def:fonction_regularisante_densite_Zk}
\end{equation}%
Then $\varphi _{v}\in \mathcal{C}_{b}^{\infty }(\mathbb{R}^N)$, $0\leqslant \varphi
_{v}\leqslant 1$ and we have the following crucial property: For every $p,q\in \mathbb{N}$
there exists a universal constant $C_{q,p}$ such that for every $z\in \mathbb{R}^N$ and $i_1, \ldots, i_q \in \{1,\ldots,N \}$, we have
\begin{equation}
\varphi _{v}(z)\vert \frac{\partial^{q}}{\partial z^{i_1} \cdot \partial z^{i_q}}(\ln \varphi _{v})(z)\vert ^{p}\leqslant 
\frac{C_{q,p}}{v^{pq}},  \label{eq:borne_deriv_log_reg_Zk}
\end{equation}%
with the convention $\ln \varphi_v(z)=0$ for $\vert z \vert \geqslant 2v$. As an immediate consequence of (\ref{hyp:lebesgue_bounded}), for every non negative function $%
f:\mathbb{R}^{N}\rightarrow \mathbb{R}_{+}$ and $t \in \pi^{\delta} \cap (0,T]$
\begin{align*}
\mathbb{E}[f(Z^{\delta}_{t})] \geqslant \varepsilon _{\ast }\int_{\mathbb{R}^{N}}\varphi _{r_{\ast
}/2}(\ z-z_{\ast ,t}\ )f(z)dz.  
\end{align*}%
By a change of variable%
\begin{align}
\mathbb{E}[f(\sqrt{\delta}Z^{\delta}_{t})] \geqslant \varepsilon _{\ast
}\int_{\mathbb{R}^{N}}\delta^{-N/2}\varphi _{r_{\ast }/2}\big(\sqrt{\delta}^{-1}( z-\sqrt{\delta}
z_{\ast ,t}) \big)f(z)dz.  \label{eq:borne_esperance_lebesgue_normalisee}
\end{align}%
We denote 
\begin{align*}
m_{\ast }=\varepsilon _{\ast }\int_{\mathbb{R}^{N}}\varphi _{r_{\ast }/2}(
z )dz=\varepsilon _{\ast }\int_{\mathbb{R}^{N}}\varphi _{r_{\ast
}/2}( z-z_{\ast ,t} )dz  
\end{align*}%
and%
\begin{align*}
\phi _{\delta}(z)=\delta^{-N/2}\varphi _{r_{\ast }/2}(\frac{1}{\sqrt{\delta}} z)
\end{align*}%
and we notice that $\int \phi _{\delta}(z)dz=m_{\ast }\varepsilon _{\ast }^{-1}.$

We consider a sequence of independent random variables $\chi^{\delta}_{t}\in
\{0,1\},\; U^{\delta}_{t}, V^{\delta}_{t}\in \mathbb{R}^{N}$, $t \in \pi^{\delta} \cap (0,T]$, with laws given by%

\begin{eqnarray}
\mathbb{P}(\chi^{\delta} _{t} &=&1)=m_{\ast },\qquad \mathbb{P}(\chi^{\delta}_{t}=0)=1-m_{\ast },  \label{def:loi_Chik_Uk_Vk}
\\
\mathbb{P}(U^{\delta}_{t} &\in &dz)=\frac{\varepsilon _{\ast }}{m_{\ast }}\phi _{\delta}(z-\sqrt{\delta} 
z_{\ast ,t})dz,  \nonumber \\
\mathbb{P}(V^{\delta}_{t} &\in &dz)=\frac{1}{1-m_{\ast }}(\mathbb{P}(\sqrt{\delta}Z^{\delta}_{t}\in
dz)-\varepsilon _{\ast }\phi _{\delta}(z-\sqrt{\delta} z_{\ast ,t})dz). 
\nonumber
\end{eqnarray}%

Notice that (\ref{eq:borne_esperance_lebesgue_normalisee}) guarantees that $\mathbb{P}(V^{\delta}_{t}\in dz)\geqslant 0.$\ \ Then a
direct computation shows that 
\begin{align}
\mathbb{P}(\chi^{\delta}_{t}U^{\delta}_{t}+(1-\chi^{\delta}_{t})V^{\delta}_{t}\in dz)=\mathbb{P}(\sqrt{\delta}Z^{\delta}_{t}\in dz).
\label{eq:egalite_loiZk_Splitting_Numelin}
\end{align}%
This is the splitting procedure for $\sqrt{\delta}Z^{\delta}_{t}$. Now on we
will work with this representation of the law of $\sqrt{\delta}Z^{\delta}_{t}.$
So, we always take 
\begin{align*}
\sqrt{\delta}Z^{\delta}_{t}=\chi^{\delta}_{t}U^{\delta}_{t}+(1-\chi^{\delta}_{t})V^{\delta}_{t}. 
\end{align*}

\begin{remark}
The above splitting procedure has already been widely used in the
litterature: In \cite{Nummelin_1978} and \cite{Locherbach_Loukianova_2008}, it is used in order to prove convergence to
equilibrium of Markov processes. In \cite{Bobkov_Chistyakov_Gennadiy_Gotze_2014_BerryEsseen}, \cite{Bobkov_Chistyakov_Gennadiy_Gotze_2014_FisherCLT} and \cite{Yu_Zaitsev_1995}, it is used to
study the Central Limit Theorem. Last but not least, in \cite{Nourdin_Poly_2013}, the above
splitting method (with $\mathds{1}_{B_{r_{\ast }}(z_{\ast ,t})}$ instead of $\phi
_{\delta}(z-\frac{z_{\ast ,t}}{\sqrt{n}}))$ is used in a framework which is
similar to the one in this paper.
\end{remark}

%

\subsection{The regularization property}

In the following, we will not work under $\mathbb{P}$, but under a localized
probability measure which we define now. For $t >0$, let
\begin{equation}
\Lambda^{\delta} _{t}=\left\{\frac{1}{t}\sum_{s\in \pi^{\delta} \cap (0,t]}\chi^{\delta}_{s}\geqslant \frac{m_{\ast }}{2} \right\}.
\label{def:espace_localisation_chi}
\end{equation}%
Using Hoeffding's inequality and the fact that $\mathbb{E}[\chi^{\delta}_{t}]=m_{\ast }$, it can be checked that
\begin{equation}
\mathbb{P}(\Omega \setminus \Lambda^{\delta} _{t})\leqslant  \exp(-m_{\ast }^{2} \lfloor t / \delta \rfloor /2) .  \label{eq:proba_hors_espace_localisation_chi}
\end{equation}
We consider also the localization function $\varphi _{n^{1/4}/2}$, defined in (%
\ref{def:fonction_regularisante_densite_Zk}), and we construct the random variable%

\begin{equation}
\Theta ^{\delta} _{t}=\mathds{1}_{\Lambda^{\delta} _{t}}\times \prod_{t \in \pi^{\delta} \cap (0,T]}\varphi
_{\delta^{-1/4}/2}(Z^{\delta}_{t}).  \label{def:espace_localisation_general}
\end{equation}
Since $Z^{\delta}_{t}$ has finite moments of any order, the following inequality can be shown: For
every $u \in \mathbb{N}$, we have
\begin{equation}
\mathbb{P}(\Theta ^{\delta} _{t}=0)\leqslant \mathbb{P}(\Omega \setminus \Lambda^{\delta} _{t})+\sum_{t \in \pi^{\delta} \cap (0,T]}\mathbb{P}(\vert
Z^{\delta}_{t}\vert \geqslant \delta^{-1/4})\leqslant \exp(-m_{\ast }^{2} \lfloor t/\delta \rfloor /2) + \delta^u M_{4(u+1)}(Z^{\delta}).
\label{eq:proba_hors_Loc}
\end{equation}


We define the probability measure%
\begin{equation}
d\mathbb{P}_{\Theta }=\frac{1}{\mathbb{E}[\Theta ]}\Theta d\mathbb{P}.  \label{def:proba_localisee}
\end{equation}

We consider the Markov chain $(X^{\delta}_t)_{t \in \pi^{\delta}}$, defined in (\ref{eq:schema_general}) and we introduce $(Q_{t}^{\delta,\Theta^{\delta}})_{t \in \pi^{\delta}}$ such that,
\begin{align}
\label{def:semigroupe_regularisant}
\forall t \in \pi^{\delta}, \quad Q^{\delta,\Theta^{\delta}}_{t}f(x):=\mathbb{E}_{\Theta ^{\delta} _{t} }[f(X^{\delta}_{t}(x))]=\frac{1}{\mathbb{E}[\Theta ^{\delta} _{t}]}\mathbb{E}[\Theta ^{\delta} _{t}
f(X^n_{t}(x))].
\end{align}
 Notice that $(Q^{\delta,\Theta }_{t})_{t \in \pi^{\delta}}$, is not a semigroup, but this
is not necessary. We will not be able to prove the regularization property for $Q^{\delta}$ but for $Q^{\delta,\Theta}$. The results we give are proved in \cite{Bally_Rey_2016} (see Proposition 4.5 and Corollary 4.7) and we simply adapt their notations to our framework.
\begin{proposition}
\label{prop:regularisation}
\begin{enumerate}[label=\textbf{\Alph*.}]
\item Let $\lambda _{\ast }>0$ and assume that
\begin{equation}
\inf_{\kappa \in \mathbb{R}}\inf_{x\in \mathbb{R}^{d}}\inf_{\vert \xi \vert
=1}\sum_{i=1}^{N}\left\langle \partial _{z_{i}}\psi (\kappa ,x,0,0),\xi
\right\rangle ^{2}\geqslant \lambda _{\ast } .
 \label{hyp:borne_deriv_fonction_schema}
\end{equation}
Moreover, assume that $\delta >0$ and $t \in \pi^{\delta}$ satisfy
\begin{align}
3 \delta^{1/4} \Vert \psi \Vert _{1,3,\infty
}+\delta M_{8}(Z^{\delta})+\exp (-m_{\ast }^{2}  t /(2 \delta))\leqslant \frac{1}{2}
\label{hyp:borne_flot_tangent}
\end{align}
%
and
\begin{align}
\label{hyp:ellipt_generale_asymp}
\delta^{-1/2} \geqslant   \frac{{8 (N^3+N^2+1)}}{\lambda_{\ast}} \Vert \psi \Vert_{1,3, \infty}^2. 
\end{align}


Then for every $q \in \mathbb{N}$ and multi index $\alpha ,\beta $ with $\vert \alpha
\vert +\vert \beta \vert \leqslant q$, there exists $l\in \mathbb{N}^*$ and $C \geqslant 1$ which depend on $m_{\ast},r_{\ast}$ and the moments of $Z$ such that
\begin{eqnarray}
 \Vert \partial _{\alpha }Q^{\delta,\Theta^{\delta}_{t}}_{t}\partial _{\beta }f\Vert
_{\infty } &\leqslant & C  \frac{\mathfrak{K}_{q+3}(\psi)^l}{(
\lambda _{\ast }t)^{q(q+1)}}\Vert f\Vert _{\infty }  \label{eq:borne_semigroupe_regularisation} 
\end{eqnarray}%
with $\mathfrak{K}_r(\psi)$ defined in (\ref{eq:Constante_gronwall_fonction_schema}). In particular, $Q^{\delta,\Theta^{\delta}_{t}}_{t}(x,dy) =p_{t}^{n,\Theta^{\delta}_{t}}(x,y)dy$ and $(x,y) \mapsto p_{t}^{n,\Theta^{\delta}_{t}}(x,y)$ belongs to $\mathcal{C}^{\infty}(\mathbb{R}^d \times \mathbb{R}^d)$. \\
\item  For every $u \in \mathbb{N}$ and $t \in \pi^{\delta}$, we have%
\begin{equation}
\Vert Q^{\delta}_{t}f-Q^{\delta,\Theta^{\delta}_{t}}_{t}f\Vert _{\infty }\leqslant  4 (  \exp({-m_{\ast}^2}t/(2 \delta))+\delta^u M_{4(u+1)}(Z^{\delta})) \Vert f\Vert _{\infty }.  \label{eq:convergence_modification}
\end{equation}
\end{enumerate}
\end{proposition}

\begin{proposition}
\label{prop:regularisation}
\begin{enumerate}[label=\textbf{\Alph*.}]
\item Let $\lambda _{\ast }>0$ and assume that
\begin{equation}
\inf_{\kappa \in \mathbb{R}}\inf_{x\in \mathbb{R}^{d}}\inf_{\vert \xi \vert
=1}\sum_{i=1}^{N}\left\langle \partial _{z_{i}}\psi (\kappa ,x,0,0),\xi
\right\rangle ^{2}\geqslant \lambda _{\ast },
 \label{hyp:borne_deriv_fonction_schema}
\end{equation}
Then, there exists $\delta_0$ depending on $\Vert \psi \Vert_{1,3, \infty}$, $m_{\ast}$ and $\lambda_{\ast}$ and such that for every $\delta \leqslant \delta_0$ and $t \in \pi^{\delta}$ satisfying
\begin{align}
exp (-m_{\ast }^{2}  t /(2 \delta))\leqslant \frac{1}{2},
\label{hyp:borne_flot_tangent}
\end{align}
we have: for every $q \in \mathbb{N}$ and multi index $\alpha ,\beta $ with $\vert \alpha
\vert +\vert \beta \vert \leqslant q$, there exists $l\in \mathbb{N}^*$ and $C \geqslant 1$ which depend on $m_{\ast},r_{\ast}$ and the moments of $Z^{\delta}$ such that, 
\begin{eqnarray}
 \Vert \partial _{\alpha }Q^{\delta,\Theta^{\delta}_{t}}_{t}\partial _{\beta }f\Vert
_{\infty } &\leqslant & C  \frac{\mathfrak{K}_{q+3}(\psi)^l}{(
\lambda _{\ast }t)^{q(q+1)}}\Vert f\Vert _{\infty }  \label{eq:borne_semigroupe_regularisation} 
\end{eqnarray}%
with $\mathfrak{K}_r(\psi)$ defined in (\ref{eq:Constante_gronwall_fonction_schema}). In particular, $Q^{\delta,\Theta^{\delta}_{t}}_{t}(x,dy) =p_{t}^{n,\Theta^{\delta}_{t}}(x,y)dy$ and $(x,y) \mapsto p_{t}^{n,\Theta^{\delta}_{t}}(x,y)$ belongs to $\mathcal{C}^{\infty}(\mathbb{R}^d \times \mathbb{R}^d)$. \\
\item  For every $u \in \mathbb{N}$ and $t \in \pi^{\delta}$, we have%
\begin{equation}
\Vert Q^{\delta}_{t}f-Q^{\delta,\Theta^{\delta}_{t}}_{t}f\Vert _{\infty }\leqslant  4 (  \exp({-m_{\ast}^2}t/(2 \delta))+\delta^u M_{4(u+1)}(Z^{\delta})) \Vert f\Vert _{\infty }.  \label{eq:convergence_modification}
\end{equation}
\end{enumerate}
\end{proposition}

\begin{remark}
(\ref{eq:borne_semigroupe_regularisation}) means that the strong regularization property $R_{q, \eta}(Q^{\delta,\Theta^{\delta}})$ (see (\ref{hyp:reg_forte})) holds with $\eta(q)=q(q+1)$. Notice also that usually the moments of $Z^{\delta}$ do not depend on $\delta$ so $C$ does not depend on $\delta$.
\end{remark}

We give now an alternative way to regularize the semigroup $Q^{\delta}$ (by
convolution). We consider a $d$ dimensional standard normal random variable 
$G$ which is independent from $Z^{\delta}_{t},t\in \pi^{\delta} \cap (0,T]$, and for $\theta >0$, we introduce $(X^{\delta,\theta}_t)_{t \in \pi^{\delta}}$ as follows
\begin{align}
\label{def:scheme_convolution}
X_{t}^{\delta,\theta }(x)=\delta^{\theta}G +X^{\delta}_{t}(x).
\end{align}
We denote by $p_{t}^{\delta,\theta}(x,y)$ the density of the law of $X^{\delta,\theta}_{t}(x)$ and for $t \in \pi^{\delta}$, we define
\begin{align}
\label{def:semigroup_convolution}
Q^{\delta,\theta }_{t}f(x):=\mathbb{E}[f(\delta^{\theta} G +X^{\delta}_{t}(x))].
\end{align}

\begin{coro}
\label{coro:regul_semigroup_convol}Under the hypothesis of the previous proposition we have:
\begin{enumerate}[label=\textbf{\Alph*.}]
\item For every
multi index $\alpha ,\beta $ with $\vert \alpha \vert
+\vert \beta \vert \leqslant q$, and every $q \in \mathbb{N}^*$, there exists $l \in \mathbb{N}^{\ast}$, $
C\geqslant 1$, which depend on $q$, $T$ and the moments of $Z$ such that for all $l' \in \mathbb{N}$ and $t \in \pi^{\delta}$ such that
have (\ref{hyp:borne_flot_tangent}) hold, the following estimate holds :%
\begin{equation}
\Vert \partial _{\alpha }Q^{\delta,\theta }_{t}\partial _{\beta }f\Vert
_{\infty }\leqslant C  \Big( \frac{\mathfrak{K}_{q+3}(\psi)^l }{%
(\lambda _{\ast } t)^{q(q+1)} }+\delta^{-q \theta } \mathfrak{K}_{q+3}(\psi)^l  (\exp({-m_{\ast}^2 t /(4 \delta)})+\delta^{l'/2} M_{4(l'+1)}(Z^{\delta})^{1/2}) \Big) \Vert f\Vert _{\infty } , \label{eq:regul_semigroup_convol}
\end{equation}%
with $\mathfrak{K}_r(\psi)$ defined in (\ref{eq:borne_semigroupe_regularisation}).\\
\item There exists $l \in \mathbb{N}^{\ast}$, $C\geqslant 1$, such that for every $l' \in \mathbb{N}$ and $t \in \pi^{\delta}$
\begin{equation}
 \Vert Q^{\delta}_{t}f(x)-Q^{\delta,\theta }_{t}f(x)\Vert_{\infty}  \leqslant C \big( \delta^{\theta}  \frac{\mathfrak{K}_4(\psi)^l}{(\lambda_{\ast} t)^2 }+2  ( \exp({-m_{\ast}^2  t /(2\delta)})+\delta^{l'} M_{4(l'+1)}(Z^{\delta})) \big) \Vert f \Vert_{\infty}.  \label{eq:convergence_modification}
\end{equation}
\end{enumerate}
\end{coro}

\subsection{Total variation convergence result}

We recall that $T >0$ and $n\in \mathbb{N}$ are fixed. In this section we give the approximation result for a (homogeneous) Markov semigroup $%
(P_{t})_{t\geqslant 0}$ at time $T$. We first introduce our discrete semigroups. For $\delta >0$, we denote $\mu_{t}(x,dy)=P_{\delta}(x,dy)$ for all $t>0$. We
consider now an approximation scheme based on the Markov chain introduced in
in this section We assume that $Z^{\delta}_{\delta}, \ldots ,Z^{\delta}_T$
verifies (\ref{hyp:lebesgue_bounded}) and have finite moments of
any order: For every $\delta>0$, $p \geqslant 1$,
\begin{align}
\label{eq:borne_moment_Z}
 M_{p}(Z^{\delta})=1\vee \sup_{t \in (0,T] \cap \pi^{\delta} }\mathbb{E}[\vert Z^{\delta}_{t}\vert ^{p}]<\infty .
\end{align}

 Moreover, using the representation (\ref{eq:schema_general}) we take $\psi \in \mathcal{C}^{\infty
}(\mathbb{R} \times \mathbb{R}^{d}\times \mathbb{R}^{N} \times \mathbb{R}_+;\mathbb{R}^{d})$ such that $\psi(\kappa ,x,0,0)=x$ and
we construct $X^{\delta}_{t+\delta}=\psi (\kappa _{t},X^{\delta}_{t},\sqrt{\delta} Z^{\delta}_{t+\delta}, \delta)$ for every $t \in [0,T-\delta] \cap \pi^{\delta}$. Its transition measures are given by $\nu
_{t}^{\delta}(x,dy)=\mathbb{P}(X^{\delta}_{t+\delta }\in dy\mid X^{\delta}_{t}=x)$, $t \in \pi^{\delta}$ and we
construct the discrete semigroup $Q^{\delta}_{t+\delta}=Q_{t}^{\delta}\nu _{t}^{\delta}$ on the time grid $\pi^{\delta}$. We
recall that the notation $\Vert \psi \Vert _{1,r,\infty }$ is
introduced in (\ref{eq:Norme_adhoc_fonction_schema}) and we assume that, for every $r \in \mathbb{N}$,
\begin{equation}
\Vert \psi \Vert _{1,r,\infty }<\infty .  \label{eq:borne_fonction_schema_approx_result}
\end{equation}%

We also assume that there exists $\lambda _{\ast }>0$ such that 
\begin{equation}
\inf_{\kappa \in R}\inf_{x\in \mathbb{R}^{d}}\inf_{\vert \xi \vert
=1}\sum_{i=1}^{N}\left\langle \partial _{z_{i}}\psi(\kappa ,x,0,0),\xi
\right\rangle ^{2}\geqslant \lambda _{\ast }.  \label{eq:borne_fonction_inf_schema_approx_result}
\end{equation}%

Using the family of approximation schemes $(X^{\delta})_{\delta>0}$, for every $\nu \in \mathbb{N}^{\ast}$, we consider $\hat{Q}^{\nu,\delta_{n}^{0}}_{0,T}$ defined as in (\ref{eq:schema_ordre_quelconque}). Now we are able to prove our main result.

\begin{theorem}
\begin{enumerate}[label=\textbf{\Alph*.}]
Let $T>0$, $n,\nu \in \mathbb{N}^{\ast}$ and $q_{\nu}=\max_{i\in \{1,\ldots,m(0,\nu) \}}(i\max(\beta,\kappa(1,q_i(\nu,0)))$ with $n$ large enough so the hypothesis from Proposition \ref{prop:regularisation} hold with $\delta$ replaced by $\delta^{0}_n$.
\item We assume that (\ref{hyp:transport_regularite_semigroup}), (\ref{hyp:transport_regularite_semigroup_dual}), (\ref{eq:borne_moment_Z}), (\ref{eq:borne_fonction_schema_approx_result}) and (\ref{eq:borne_fonction_inf_schema_approx_result}) hold. Moreover, we assume that for every $k \in \mathbb{N}$, $k \geqslant n$, (\ref{hyp:transport_regularite_semigroup_approx}) and (\ref{hyp:transport_regularite_semigroup_approx_dual}) hold with $n$ replaced by $k$ and that the short time estimates $E_k(l,\alpha,\beta,P,Q)$ (see (\ref{hyp:erreur_tems_cours_fonction_test_reg})), and $E_k(l,\alpha,\beta,P,Q)^{\ast}$ (see (\ref{hyp:erreur_tems_cours_fonction_test_reg_dual})) hold for every $l \in \{1,\ldots,l(\nu,\alpha)\}$ if $k=n$ and for $l=l(\nu,\alpha)$ if $k > n$. Then, there exists $l \in \mathbb{N}^{\ast}$, $C\geqslant 1$, which depend on $q_{\nu}$, $T$ and the moments of $(Z^{\delta})_{\delta >0}$, such that
\begin{equation}
\Vert P_{T}f-\hat{Q}^{\nu,\delta_{n}^{0}}_{0,T}f\Vert
_{\infty }\leqslant C \frac{\mathfrak{K}_{q_{\nu}+3}(\psi)^l}{(
\lambda _{\ast }T(\nu))^{\eta(q_{\nu})}}\Vert f\Vert _{\infty } \frac{1}{n^{\nu}}.  \label{eq:distance_convergence_total_variation}
\end{equation}
with $\eta(q)=q(q+1)$.

\item We have also, $P_{T}(x,dy)=p_{T}(x,y)dy$ and $\hat{Q}^{\nu,\delta_{n}^{0}}_{0,T}=p_{T}^{n,\nu}(x,y)dy$ with $(x,y)\mapsto p_{T}(x,y)$ and $(x,y)\mapsto p_{T}^{n,\nu}(x,y)$ belonging to $\mathcal{C}^{\infty}(\mathbb{R}^{d}\times \mathbb{R}^{d})$.\\

Moreover, for every $R>0, \varepsilon, \hat{\varepsilon} \in (0,1)$, and every multi-index $\gamma$, $\xi$, there exists $l \in \mathbb{N}^{\ast}$, $C\geqslant 1$, which depend on $q_{\nu}$, $T$ and the moments of $(Z^{\delta})_{\delta >0}$, such that, 
\begin{align}
\sup_{(x,y) \in \overline{B}_R(x_0,y_0)}\vert \partial _{x}^{\gamma }\partial _{y}^{\xi
}p_{T}(x,y)-  \partial _{x}^{\gamma }\partial _{y}^{\xi }p_{T}^{n,\nu}(x,y)\vert \leqslant  \frac{C \mathfrak{K}_{q_{\nu}+3}(\psi)^l}{ (\lambda _{\ast } T(\nu))^{\eta (p(\nu,\vert
\gamma \vert +\vert \xi \vert,\varepsilon,\hat{\varepsilon}) )}} \frac{1}{n^{\nu (1- \varepsilon)}} \label{eq:distance_density}  
\end{align}

with a constant $C$ which depends on $R,x_0,y_0,T $ and on $\vert
\gamma \vert +\vert \xi \vert $ and $p(\nu,u,\varepsilon,\hat{\varepsilon}) = p_{p_{u,\hat{\varepsilon}} \vee q_{\nu}-2d,\varepsilon} \vee q_{\nu}$ with  $p_{u,\varepsilon}= (u+2d+1+2\lceil  (1- \varepsilon)(u+d)/(2\varepsilon)  \rceil).$

\end{enumerate}

\end{theorem}

\begin{proof}
We have proved in Proposition \ref{prop:regularisation} that $Q^{\delta,\Theta^{\delta}}$ verifies the regularization
properties (we could similarly use the modification introduced in Corollary \ref{coro:regul_semigroup_convol}). The proof of (\ref{eq:distance_convergence_total_variation}) and (\ref{eq:distance_density}) is then an immediate consequence  of Theorem \ref{theo:distance_density}. 
\end{proof}

\bibliography{Biblio}
\bibliographystyle{plain}

\end{document}